\documentclass[pdflatex,sn-mathphys-num]{sn-jnl}

\usepackage{graphicx}%
\usepackage{multirow}%
\usepackage{amsmath,amssymb,amsfonts}%
\usepackage{amsthm}%
\usepackage{mathrsfs}%
\usepackage[title]{appendix}%
\usepackage{xcolor}%
\usepackage{textcomp}%
\usepackage{manyfoot}%
\usepackage{booktabs}%
\usepackage{algorithm}%
\usepackage{algorithmicx}%
\usepackage{algpseudocode}%
\usepackage{listings}%

\newtheorem{theorem}{Theorem}[section]
\newtheorem{proposition}[theorem]{Proposition}
\newtheorem{lemma}[theorem]{Lemma}

\newtheorem{definition}[theorem]{Definition}
\newtheorem{assumption}[theorem]{Assumption}
\newtheorem{remark}[theorem]{Remark}

\begin{document}

\title{Rigorous Analysis of a Nonlocal Transport--Renewal System for Physiologically Structured Populations}

\author[1]{\fnm{Jiguang} \sur{Yu}}\email{jyu678@bu.edu}
\equalcont{These authors contributed equally to this work as co-first authors.}

\author*[2]{\fnm{Louis Shuo} \sur{Wang}}\email{wang.s41@northeastern.edu}
\equalcont{These authors contributed equally to this work as co-first authors.}

\author*[3]{\fnm{Ye} \sur{Liang}}\email{ye-liang@uiowa.edu}
\equalcont{These authors contributed equally to this work as co-first authors.}

\affil[1]{\orgdiv{College of Engineering}, \orgname{Boston University}, \orgaddress{\city{Boston}, \state{MA}, \postcode{02215}, \country{USA}}}

\affil[2]{\orgdiv{Department of Mathematics}, \orgname{Northeastern University}, \orgaddress{\city{Boston}, \state{MA}, \postcode{02115}, \country{USA}}}

\affil[3]{\orgdiv{College of Engineering},
  \orgname{The University of Iowa},
  \orgaddress{\city{Iowa City}, \postcode{52242}, \state{IA},\country{USA}}}

\abstract{We develop a rigorous analytical framework for a class of physiologically structured population models with two internal state variables, nonlocal ecological feedbacks, dynamic resources, inter-zone transfer, and selective harvesting. The full model is a coupled nonlinear PDE--ODE transport--renewal system with endogenous inflow at the recruitment boundary, a setting in which transport, nonlocal dependence, and boundary renewal interact at the same level. For this full nonautonomous multi-zone system, we prove finite-horizon well-posedness in a positive $L^{1}$-based state space, including global existence on arbitrary bounded time intervals, uniqueness, nonnegativity, and continuous dependence on initial data, environmental forcing, and harvesting effort. We then isolate an autonomous single-zone reduction at extinction and construct a positive compact next-generation operator on the recruit space. 
In a further nonlinear stationary reduction, we prove that supercriticality of the basic reproduction number $\mathcal R_{0}>1$ yields existence of a nontrivial stationary state under a parametrized compact-operator hypothesis encoding density-dependent renewal feedback. Finally, for a finite-horizon harvest objective over a compact Lipschitz-regular admissible class, we establish existence of an optimal control. The results separate what can be proved for the full climate-explicit system from what can be justified only after autonomous reduction, thereby clarifying the mathematical scope of threshold and control theory for structured populations.}



\maketitle

\section{Introduction}\label{sec:introduction}
Classical structured population theory starts from age- and size-structured transport
equations of McKendrick, von Foerster, Sinko, Streifer, and their nonlinear
extensions; in modern form this leads to physiologically structured population
models, renewal equations, and operator-theoretic threshold quantities; see among many sources \cite{McKendrick1926,wang2026elliptic,VonFoerster1959,SinkoStreifer1967,Webb1985,yu2026chemotactic,vargelouglu2026efficient,MetzDiekmann1986,Iannelli1995,Cushing1998,wang2025analysis,monforte2022evaluation,MagalRuan2018,DiekmannEtAl2010,wang2026damage,murray2002mathematical,hu2024stability,wang2021global,wu2014spectral,wang2025analysis1}. Let
\[
\Omega_l=[l_0,l_m],\qquad \Omega_c=[0,c_m],\qquad \mathcal S=\{1,\dots,S\},
\]
with \(0<l_0<l_m<\infty\), \(c_m>0\), and \(S\in\mathbb N\). For each zone
\(s\in\mathcal S\), let
\[
x_s=x_s(t,l,c)\ge 0,
\qquad
(t,l,c)\in [0,\infty)\times\Omega_l\times\Omega_c,
\]
denote the density of individuals of structural size \(l\) and physiological
condition \(c\). The state space is
\[
X:=\prod_{s=1}^S L^1(\Omega_l\times\Omega_c),
\qquad
X_+:=\prod_{s=1}^S L^1_+(\Omega_l\times\Omega_c),
\qquad
\|x\|_X:=\sum_{s=1}^S\|x_s\|_{L^1(\Omega_l\times\Omega_c)}.
\]
For the class of populations motivating the present work, a one-dimensional
structuring variable is not sufficient. Individuals with the same structural size may
have substantially different energetic reserves, reproductive readiness, or stress
tolerance. This leads naturally to a two-dimensional physiological state
\((l,c)\in\Omega_l\times\Omega_c\).
In addition, mortality, recruitment, and harvesting are not purely local in the state
variable: they depend on nonlocal population summaries, environmental forcing, and
management effort. The resulting state equation is a transport--renewal system with
coupled lower-dimensional feedback.
The full model studied here consists of a coupled PDE--ODE system of the form
\begin{equation}\label{eq:intro_state}
\begin{cases}
\displaystyle
\partial_t x_s
+\partial_l\!\big(g_s(\mathcal N_s[x],R_s,Y,l,c)\,x_s\big)
+\partial_c\!\big(h_s(\mathcal N_s[x],R_s,Y,l,c)\,x_s\big)
\\[1mm]
\hspace{2.6cm}
=
-\Big(\mu_s(\mathcal N_s[x],R_s,Y,l,c)+u_s^{(q)}(t,l,c)\Big)x_s
+\mathcal M_s[x](l,c),
\\[2mm]
\displaystyle
\frac{d}{dt}R_s(t)
=
F_s(R_s(t),Y(t))
-\rho_s(R_s(t))
\int_{\Omega_l}\int_{\Omega_c}\kappa_s(l,c)x_s(t,l,c)\,dc\,dl,
\end{cases}
\end{equation}
for \(s\in\mathcal S\), together with the endogenous inflow boundary condition
\begin{equation}\label{eq:intro_bc}
g_s(\mathcal N_s[x(t)],R_s(t),Y(t),l_0,c)\,x_s(t,l_0,c)
=
\mathcal B_s[x(t),R_s(t),Y(t)](c),
\qquad c\in\Omega_c.
\end{equation}
Here
\(\mathcal N_s[x]\in\mathbb R_+^m\)
collects nonlocal ecological feedbacks,
\(\mathcal M_s[x]\)
is the zone-coupling operator,
\(R_s\)
is a resource or habitat variable,
and
\(u_s^{(q)}(t,l,c)=q_s(t)\sigma_s(l,c)\eta_s(\widehat B_s(t),Y(t))\)
is the selective harvesting mortality. The boundary condition
\eqref{eq:intro_bc} is of renewal type: recruitment is not prescribed externally, but
is generated by the current population state.
The mathematical difficulty is therefore threefold:
\begin{equation}\label{eq:intro_difficulties}
\text{transport in }(l,c)
\quad+\quad
\text{nonlocal nonlinear feedback}
\quad+\quad
\text{endogenous boundary inflow}.
\end{equation}
Only the first of these is handled at the full-system level; the second and third are
treated spectrally only after autonomous reduction.

The purpose of this paper is narrower and fully analytic. We isolate a model class for
which the following assertions can be proved rigorously.
\begin{enumerate}[label=\textbf{(C\arabic*)}]
\item \textbf{Finite-horizon well-posedness of the full coupled PDE--ODE system.}
For every admissible environmental path \(Y\), effort control \(q\), and initial data
\((\phi,R_0)\in X_+\times\mathbb R_+^S\), the system
\eqref{eq:intro_state}--\eqref{eq:intro_bc} admits a unique nonnegative weak solution
\[
(x,R)\in
\Big(C([0,T];X)\cap L^\infty((0,T)\times\Omega_l\times\Omega_c)^S\Big)
\times
\Big(W^{1,\infty}(0,T;\mathbb R^S)\cap C([0,T];\mathbb R_+^S)\Big),
\]
with continuous dependence on data. 
\(L^\infty\)-continuous dependence requires the
stronger assumption that the initial data converge also in \(L^\infty\).

\item \textbf{A rigorous reduced threshold operator.}
After freezing the environment at \(Y(t)\equiv y^\ast\), fixing constant effort
\(q(t)\equiv \bar q\), and restricting to a single-zone autonomous reduction, we
construct a positive compact next-generation operator
\[
\mathcal K:L^1(\Omega_c)\to L^1(\Omega_c),
\]
and define \(\mathcal R_0(y^\ast,\bar q):=r(\mathcal K)\).
In that reduced setting, \emph{conditional on an explicit semigroup embedding hypothesis},
\(\mathcal R_0(y^\ast,\bar q)\) determines the sign of the
spectral bound of the linearized transport--renewal generator.

\item \textbf{Existence of stationary states under a parametrized nonlinear
renewal reduction.}
Under an explicit nonlinear operator hypothesis (a parametrized family of compact
operators encoding the stationary feedback structure), supercriticality
\(\mathcal R_0(y^\ast,\bar q)>1\)
implies existence of at least one nontrivial stationary state
\((x^\ast,R^\ast)\in X_+\times\mathbb R_+\), \(x^\ast\not\equiv 0\).

\item \textbf{Existence of an optimal harvesting control over a compact regular
admissible class and related topics} \cite{ragusa2008commutators,gurtin1981optimal,guariglia2021fractional,murphy1990optimal,abbas2021hybrid,brokate1987certain,guariglia2022fractional,wang2026breakdown,brokate1985pontryagin,alotaibi2024absence,liu2026computational}.
For a finite-horizon objective functional \(J\), maximized over a compact regular class
\[
\mathcal Q_{\mathrm{ad},\mathrm{reg}}^T
\subset W^{1,\infty}(0,T;\mathbb R^S),
\]
we prove existence of an optimizer \(q^\ast\). No adjoint system or Pontryagin conditions
are derived.
\end{enumerate}
We do not prove the following results for the full climate-explicit multi-zone system. These results include a global nonlinear extinction--persistence dichotomy,
a multi-zone \(\mathcal R_0\)-type threshold theorem, a complete stationary bifurcation theory, and a rigorous adjoint system or Pontryagin maximum principle.

The structure of the paper is as follows.
Section~\ref{sec:model} formulates the full climate-explicit PDE--ODE system and its
weak solution concept.
Section~\ref{sec:wellposedness} proves finite-horizon well-posedness and continuous
dependence.
Section~\ref{sec:spectral_R0} constructs the reduced autonomous next-generation
operator, establishes the threshold sign principle, and derives
local nonlinear consequences.
Section~\ref{sec:stationary} establishes existence of reduced stationary states under
an explicit nonlinear operator hypothesis.
Section~\ref{sec:optimal_harvest} proves existence of optimal harvesting controls by the
direct method. Section~\ref{sec:discussion} concludes.

\subsection{Analytical viewpoint}
The guiding principle is that the mathematically robust object for the full system is
the weak solution map
\((\phi,R_0,Y,q)\longmapsto (x,R)\),
whereas the mathematically robust threshold object arises only after autonomous
reduction and linearization at extinction. Thus the paper separates
full nonautonomous nonlinear PDE--ODE well-posedness from
reduced autonomous spectral theory.
This separation is deliberate. It permits rigorous results at each level without
asserting a stronger global theory than the present hypotheses support.

\subsection{Notation}
We write
\(
\mathbb R_+:=[0,\infty)\), \(\mathbf 1_A\) for the indicator of \(A\), and
\(a\lesssim b\)
to mean that \(a\le Cb\) for a constant \(C>0\) depending only on the standing
assumptions. For any \(\tau\in(0,T]\), we set
\[
Q_\tau:=(0,\tau)\times\Omega_l\times\Omega_c.
\]
The environmental state space has dimension \(d_Y\ge 1\), so that
\(Y\in C([0,T];\mathbb R^{d_Y})\). The reduced recruit space is
\(\mathcal V:=L^1(\Omega_c)\),
and the reduced threshold is
\(\mathcal R_0(y^\ast,\bar q)=r(\mathcal K)\).
All remaining notation is introduced in Section~\ref{sec:model}.

\section{Model formulation}\label{sec:model}
We formulate a coupled transport--renewal--resource system on a finite size interval and
a finite condition interval
\cite{MetzDiekmann1986,Webb1985,ni2024nonautonomous,liang2025global,GyllenbergWebb1990,hu2022dynamical,DiekmannGyllenberg2012,diekmann2020finite,liu2026ftu}.
The state variables are the zone densities
\[
x_s=x_s(t,l,c)\ge 0,
\qquad
(t,l,c)\in [0,\infty)\times\Omega_l\times\Omega_c,
\qquad s\in\mathcal S,
\]
and the zone resources
\(R_s=R_s(t)\ge 0\), \(t\ge 0\), \(s\in\mathcal S\).

\subsection{State space}
Fix
\(\Omega_l:=[l_0,l_m]\), \(0<l_0<l_m<\infty\);
\(\Omega_c:=[0,c_m]\), \(c_m>0\);
\(\mathcal S:=\{1,\dots,S\}\), \(S\in\mathbb N\).
For each \(t\ge 0\), define
\(x(t):=\big(x_1(t,\cdot,\cdot),\dots,x_S(t,\cdot,\cdot)\big)\).
We work on
\[
X:=\prod_{s=1}^S L^1(\Omega_l\times\Omega_c),
\qquad
X_+:=\prod_{s=1}^S L^1_+(\Omega_l\times\Omega_c),
\qquad
\|x\|_X
:=
\sum_{s=1}^S\|x_s\|_{L^1(\Omega_l\times\Omega_c)}.
\]
We also use
\(L^\infty(\Omega_l\times\Omega_c)^S
:=\prod_{s=1}^S L^\infty(\Omega_l\times\Omega_c)\).
For \(s\in\mathcal S\), define the abundance and weighted biomass
\begin{equation}\label{eq:model_zone_abundance}
N_s[x]
:=
\int_{\Omega_l}\int_{\Omega_c}x_s(l,c)\,dc\,dl,
\end{equation}
\begin{equation}\label{eq:model_zone_biomass}
B_s^w[x]
:=
\int_{\Omega_l}\int_{\Omega_c}w(l,c)x_s(l,c)\,dc\,dl,
\qquad
w\in L^\infty_+(\Omega_l\times\Omega_c).
\end{equation}

\subsection{Environmental forcing and nonlocal feedback}
Fix a finite horizon \(T>0\). The environmental path is prescribed:
\[
Y\in C([0,T];\mathbb R^{d_Y}),
\qquad d_Y\ge 1.
\]
For each zone \(s\in\mathcal S\), let
\[
\mathcal N_s[x]
=
\big(\mathcal N_{s,1}[x],\dots,\mathcal N_{s,m}[x]\big)\in\mathbb R_+^m,
\qquad m\in\mathbb N,
\]
where
\begin{equation}\label{eq:model_nonlocal_components}
\mathcal N_{s,j}[x]
:=
\sum_{r=1}^S
\int_{\Omega_l}\int_{\Omega_c}
\chi_{s,j,r}(l,c)\,x_r(l,c)\,dc\,dl,
\qquad j=1,\dots,m.
\end{equation}
Assume
\(\chi_{s,j,r}\in L^\infty_+(\Omega_l\times\Omega_c)\) for all
\(s,r\in\mathcal S\), \(j=1,\dots,m\).
Hence each \(\mathcal N_{s,j}:X\to\mathbb R_+\) is a bounded positive linear functional.

\subsection{Resource dynamics}
For each \(s\in\mathcal S\), the resource variable satisfies
\begin{equation}\label{eq:model_resource_ode}
\frac{d}{dt}R_s(t)
=
F_s(R_s(t),Y(t))
-
\rho_s(R_s(t))
\int_{\Omega_l}\int_{\Omega_c}
\kappa_s(l,c)x_s(t,l,c)\,dc\,dl.
\end{equation}
Assume
\(F_s:\mathbb R_+\times\mathbb R^{d_Y}\to\mathbb R\),
\(\kappa_s\in L^\infty_+(\Omega_l\times\Omega_c)\),
\(\rho_s\in C^1(\mathbb R_+;\mathbb R_+)\),
and
\begin{equation}\label{eq:model_rho_assumptions}
0\le \rho_s(R)\le \rho_{s,\max},
\qquad
\rho_s(0)=0,
\qquad
R\ge 0,
\end{equation}
\begin{equation}\label{eq:model_F_nonnegativity}
F_s(0,y)\ge 0
\qquad
\forall\,y\in\mathbb R^{d_Y}.
\end{equation}

\begin{remark}[Forward invariance of \(R_s\ge 0\)]\label{rem:resource_invariance}
The half-line \(\mathbb R_+\) is forward invariant for \eqref{eq:model_resource_ode}
whenever \(x_s\ge 0\). Indeed, at the boundary \(R_s=0\) the Nagumo \cite{Nagumo1942,gao2022rolling} inward-pointing
condition reads
\[
\dot R_s\big|_{R_s=0}
=
F_s(0,Y(t))-\underbrace{\rho_s(0)}_{=\,0}
\int\kappa_s x_s\,dc\,dl
=
F_s(0,Y(t))\ge 0,
\]
by \eqref{eq:model_rho_assumptions}--\eqref{eq:model_F_nonnegativity}. In the coupled
system, the positivity of \(x_s\) is itself part of the well-posedness conclusion; both
are established simultaneously in the fixed-point argument of
Section~\ref{sec:wellposedness}.
\end{remark}

\subsection{Vital rates}
For each \(s\in\mathcal S\), let
\[
g_s,h_s:\mathbb R_+^m\times\mathbb R_+\times\mathbb R^{d_Y}\times\Omega_l\times\Omega_c\to\mathbb R,
\qquad
\mu_s:\mathbb R_+^m\times\mathbb R_+\times\mathbb R^{d_Y}\times\Omega_l\times\Omega_c\to\mathbb R_+.
\]
The functions \(g_s,h_s,\mu_s\) represent, respectively, size velocity,
condition drift, and natural mortality.
\begin{assumption}[Standing assumptions on vital rates]\label{ass:model_vital_rates}
For each \(s\in\mathcal S\):
\begin{enumerate}[label=\textup{(V\arabic*)}]
\item\label{item:V1}
\(g_s,h_s,\mu_s
\in
C^1\big(\mathbb R_+^m\times\mathbb R_+\times\mathbb R^{d_Y}\times\Omega_l\times\Omega_c\big)\).
\item\label{item:V2}
There exist constants \(0<g_{\min}\le g_{\max}<\infty\) such that
\(0<g_{\min}\le g_s(n,R,y,l,c)\le g_{\max}\)
for all admissible arguments.
\item\label{item:V3}
There exists \(h_{\max}>0\) such that
\(|h_s(n,R,y,l,c)|\le h_{\max}\)
for all admissible arguments.
\item\label{item:V4}
There exists \(\mu_{\min}\ge 0\) such that
\(\mu_s(n,R,y,l,c)\ge \mu_{\min}\)
for all admissible arguments.
\item\label{item:V5}
All first-order partial derivatives of \(g_s,h_s,\mu_s\) are locally bounded.
\end{enumerate}
\end{assumption}

\begin{remark}[Structural consequences of \(g_{\min}>0\)]\label{rem:gmin_role}
The strict positivity assumption \ref{item:V2} plays several key roles.
It ensures that \(l=l_0\) is the unique inflow boundary for the \(l\)-transport.
It also guarantees a finite maximal residence time
\(\bar\tau_{\max}:=(l_m-l_0)/g_{\min}<\infty\). In addition,
it makes the renewal boundary condition \eqref{eq:model_recruitment_bc}
well posed, since division by \(g_s(l_0,c)\ge g_{\min}>0\) is always meaningful. Finally, it controls the boundary contribution in the \(L^\infty\) a priori estimates
of Section~\ref{sec:wellposedness}.
\end{remark}

\begin{assumption}[No inflow portion on the condition boundary]
\label{ass:model_condition_boundary_orientation}
For every admissible state \((n,R,y)\), every \(l\in\Omega_l\), and every \(s\in\mathcal S\),
\[
h_s(n,R,y,l,0)\le 0,
\qquad
h_s(n,R,y,l,c_m)\ge 0.
\]
\end{assumption}

\subsection{Inter-zone transfer}
For \(r,s\in\mathcal S\), \(r\neq s\), let
\[
\theta_{rs}:\Omega_l\times\Omega_c\to\mathbb R_+,
\qquad
\theta_{rs}\in L^\infty(\Omega_l\times\Omega_c).
\]
Interpret \(\theta_{rs}(l,c)\) as the per-capita transfer rate from zone \(r\) to zone
\(s\). Define the positive off-diagonal inflow operator
\begin{equation}\label{eq:model_transfer_in}
\mathcal M_s^{\mathrm{in}}[x](l,c)
:=
\sum_{r\neq s}\theta_{rs}(l,c)x_r(l,c),
\end{equation}
the diagonal loss coefficient
\begin{equation}\label{eq:model_transfer_out_coeff}
\theta_s^{\mathrm{out}}(l,c)
:=
\sum_{r\neq s}\theta_{sr}(l,c),
\end{equation}
and the full transfer operator
\begin{equation}\label{eq:model_transfer_operator}
\mathcal M_s[x](l,c)
:=
\mathcal M_s^{\mathrm{in}}[x](l,c)-\theta_s^{\mathrm{out}}(l,c)x_s(l,c).
\end{equation}
For every \((l,c)\in\Omega_l\times\Omega_c\),
\(\sum_{s=1}^S\mathcal M_s[x](l,c)=0\).

\subsection{Selective harvesting and controls}
For \(s\in\mathcal S\), let \(q_s=q_s(t)\ge 0\) be the effort control and let
\(\sigma_s\in L^\infty(\Omega_l\times\Omega_c)\),
\(0\le \sigma_s\le 1\) a.e.,
be the selectivity.
Let
\begin{equation}\label{eq:model_perceived_biomass}
\widehat B_s[x]
:=
\int_{\Omega_l}\int_{\Omega_c}
\varpi_s(l,c)x_s(l,c)\,dc\,dl,
\qquad
\varpi_s\in L^\infty_+(\Omega_l\times\Omega_c).
\end{equation}
The harvest mortality is
\begin{equation}\label{eq:model_harvest_mortality}
u_s^{(q)}(t,l,c)
=
q_s(t)\sigma_s(l,c)\eta_s(\widehat B_s[x(t)],Y(t)),
\end{equation}
where
\(\eta_s:\mathbb R_+\times\mathbb R^{d_Y}\to\mathbb R_+\).
The superscript \((q)\) indicates dependence on the effort control
\(q=(q_1,\dots,q_S)\).

\begin{assumption}[Harvest response]\label{ass:model_harvest_response}
For each \(s\in\mathcal S\),
\(\eta_s\in C^1(\mathbb R_+\times\mathbb R^{d_Y})\) and
\(0\le \eta_s(B,y)\le \eta_{\max}\)
for all \((B,y)\in\mathbb R_+\times\mathbb R^{d_Y}\).
\end{assumption}
For \(T>0\), define
\begin{equation}\label{eq:model_admissible_controls}
\mathcal Q_{\mathrm{ad}}^T
:=
\Big\{
q=(q_1,\dots,q_S)\in L^\infty(0,T;\mathbb R^S):
0\le q_s(t)\le q_{s,\max}\ \text{for a.e. }t\in(0,T),\ s\in\mathcal S
\Big\}.
\end{equation}

\subsection{Renewal operator at \texorpdfstring{\(l=l_0\)}{l=l_0}}
For each \(s\in\mathcal S\), define
\begin{equation}\label{eq:model_birth_operator}
\mathcal B_s[x,R,y](c)
:=
\int_{\Omega_l}\int_{\Omega_c}
\beta_s(\mathcal N_s[x],R,y,l,c')\,
\Pi_s(c\,|\,l,c',R,y)\,
x_s(l,c')\,dc'\,dl.
\end{equation}
Here
\(\beta_s:\mathbb R_+^m\times\mathbb R_+\times\mathbb R^{d_Y}\times\Omega_l\times\Omega_c
\to\mathbb R_+\)
is the fecundity kernel and
\(\Pi_s(\cdot\,|\,l,c',R,y)\)
is the recruit-condition distribution.
\begin{assumption}[Renewal kernel]\label{ass:model_birth_operator}
For each \(s\in\mathcal S\):
\begin{enumerate}[label=\textup{(B\arabic*)}]
\item
\(\beta_s\in
C^1\big(\mathbb R_+^m\times\mathbb R_+\times\mathbb R^{d_Y}\times\Omega_l\times\Omega_c\big)\),
\(\beta_s\ge 0\),
and \(\beta_s\) together with its first-order partial derivatives is locally bounded.
\item
The map
\((c,l,c')\longmapsto \Pi_s(c\,|\,l,c',R,y)\)
is jointly measurable on
\(\Omega_c\times\Omega_l\times\Omega_c\) for every fixed \((R,y)\),
nonnegative, and essentially bounded on
\(\Omega_c\times\Omega_l\times\Omega_c\), locally uniformly in \((R,y)\).
\item
For every admissible \((l,c',R,y)\),
\(\int_{\Omega_c}\Pi_s(c\,|\,l,c',R,y)\,dc=1\).
\end{enumerate}
\end{assumption}

\subsection{Controlled state equation}
For each \(s\in\mathcal S\), the density \(x_s\) satisfies
\begin{equation}\label{eq:model_state_equation}
\begin{aligned}
\partial_t x_s
&+\partial_l\!\big(g_s(\mathcal N_s[x],R_s,Y,l,c)\,x_s\big)
+\partial_c\!\big(h_s(\mathcal N_s[x],R_s,Y,l,c)\,x_s\big)
\\
&=
-\big(\mu_s(\mathcal N_s[x],R_s,Y,l,c)
+u_s^{(q)}(t,l,c)
+\theta_s^{\mathrm{out}}(l,c)\big)x_s
+\mathcal M_s^{\mathrm{in}}[x](l,c),
\end{aligned}
\end{equation}
for \((t,l,c)\in (0,\infty)\times(l_0,l_m)\times(0,c_m)\).

\subsection{Initial and boundary data}
Let
\begin{equation}\label{eq:model_initial_condition}
x_s(0,l,c)=\phi_s(l,c),
\qquad
\phi_s\in L^1_+(\Omega_l\times\Omega_c)\cap L^\infty(\Omega_l\times\Omega_c),
\end{equation}
and \(R_s(0)=R_{s,0}\ge 0\).
At the inflow boundary \(l=l_0\), impose the endogenous renewal condition
\begin{equation}\label{eq:model_recruitment_bc}
g_s(\mathcal N_s[x(t)],R_s(t),Y(t),l_0,c)\,x_s(t,l_0,c)
=
\mathcal B_s[x(t),R_s(t),Y(t)](c),
\qquad
t>0,\ c\in\Omega_c.
\end{equation}
No boundary condition is imposed at \(l=l_m\). By Assumption~\ref{ass:model_condition_boundary_orientation}, no inflow data are prescribed on
\(\{c=0\}\cup\{c=c_m\}\).

\subsection{Weak formulation}
Fix \(T>0\), \(q\in\mathcal Q_{\mathrm{ad}}^T\), and
\(Y\in C([0,T];\mathbb R^{d_Y})\).
For each \(s\in\mathcal S\), let
\(\varphi_s\in C^1([0,T]\times\Omega_l\times\Omega_c)\)
satisfy
\begin{equation}\label{eq:model_test_function_bc}
\varphi_s(T,\cdot,\cdot)=0,
\qquad
\varphi_s(\cdot,l_m,\cdot)=0,
\qquad
\varphi_s(\cdot,\cdot,0)=0,
\qquad
\varphi_s(\cdot,\cdot,c_m)=0.
\end{equation}

\begin{remark}[Fixed weak test space]\label{rem:test_space}
The conditions
\(\varphi_s(\cdot,\cdot,0)=\varphi_s(\cdot,\cdot,c_m)=0\)
are imposed to keep the weak test space \emph{independent of the unknown drift}
\(h_s(\mathcal N_s[x],R_s,Y,\cdot,\cdot)\). This avoids a solution-dependent
decomposition of the \(c\)-boundary into inflow and outflow parts. The resulting
formulation encodes no explicit boundary condition at \(c=0\) or \(c=c_m\). Instead,
Assumption~\ref{ass:model_condition_boundary_orientation} ensures that the \(c\)-boundary is
entirely outflow (or tangential), so no boundary data can be prescribed there, and uniqueness follows from the characteristic structure.
\end{remark}

\begin{definition}[Weak solution]\label{def:model_weak_solution}
A pair
\((x,R)\in
\big(C([0,T];X)\cap L^\infty(Q_T)^S\big)
\times
\big(W^{1,1}(0,T;\mathbb R^S)\cap C([0,T];\mathbb R_+^S)\big)\)
is a \emph{weak solution} of
\eqref{eq:model_resource_ode}, \eqref{eq:model_state_equation},
\eqref{eq:model_initial_condition}, \eqref{eq:model_recruitment_bc}
on \([0,T]\) if, for every \(s\in\mathcal S\) and every
\(\varphi_s\) satisfying \eqref{eq:model_test_function_bc},
\begin{equation}\label{eq:model_weak_formulation}
\begin{aligned}
&\int_0^T\!\int_{\Omega_l}\!\int_{\Omega_c}
x_s
\Big[
\partial_t\varphi_s
+g_s\,\partial_l\varphi_s
+h_s\,\partial_c\varphi_s
-\big(\mu_s+u_s^{(q)}+\theta_s^{\mathrm{out}}\big)\varphi_s
\Big]dc\,dl\,dt
\\
&\quad
+\int_{\Omega_l}\!\int_{\Omega_c}\phi_s\,\varphi_s(0)\,dc\,dl
+\int_0^T\!\int_{\Omega_c}
\mathcal B_s[x(t),R_s(t),Y(t)](c)\,\varphi_s(t,l_0,c)\,dc\,dt
\\
&\quad
+\int_0^T\!\int_{\Omega_l}\!\int_{\Omega_c}
\mathcal M_s^{\mathrm{in}}[x(t)](l,c)\,\varphi_s\,dc\,dl\,dt
=0,
\end{aligned}
\end{equation}
and, for every \(t\in[0,T]\),
\begin{equation}\label{eq:model_resource_integral}
R_s(t)
=R_{s,0}
+\int_0^t
\Big[
F_s(R_s(\tau),Y(\tau))
-\rho_s(R_s(\tau))
\int_{\Omega_l}\!\int_{\Omega_c}
\kappa_s x_s(\tau)\,dc\,dl
\Big]d\tau.
\end{equation}
\end{definition}

\section{Well-posedness of the coupled system}\label{sec:wellposedness}
Throughout this section, Assumptions~\ref{ass:model_vital_rates}, \ref{ass:model_condition_boundary_orientation}, \ref{ass:model_harvest_response}, and~\ref{ass:model_birth_operator}
are in force. Fix \(T>0\), \(Y\in C([0,T];\mathbb R^{d_Y})\),
\(q\in\mathcal Q_{\mathrm{ad}}^T\), and initial data
\[
\phi=(\phi_1,\dots,\phi_S)\in X_+\cap L^\infty(\Omega_l\times\Omega_c)^S,
\qquad
R_0=(R_{1,0},\dots,R_{S,0})\in\mathbb R_+^S.
\]

The proof is based on a frozen linear transport problem, a contraction argument in
\[
\mathbb F_{T_0}
:=
C([0,T_0];X)\times C([0,T_0];\mathbb R^S),
\]
equipped with the norm
\[
\|(x,R)\|_{\mathbb F_{T_0}}
:=
\sup_{t\in[0,T_0]}\|x(t)\|_X
+\sup_{t\in[0,T_0]}|R(t)|,
\]
followed by recovery of \(L^\infty\) bounds via a priori estimates, and continuation to
\([0,T]\).

\begin{remark}[Strategy for the \(L^\infty\) bound]\label{rem:Linfty_strategy}
The renewal boundary condition creates an instantaneous \(L^\infty\) contribution
proportional to \(\|\mathcal B_s\|_{L^\infty(\Omega_c)}/g_{\min}\).
This term does not
vanish as the time interval shrinks, so the fixed-point contraction is performed in the
weaker norm \(\|\cdot\|_{\mathbb F_{T_0}}\), which controls \(L^1\) and the resource
variable). 
The \(L^\infty\) bound is then recovered \emph{after} the fixed point is
obtained, using the a priori estimates of Lemmas~\ref{lem:wp_L1}--\ref{lem:wp_Linfty}.
\end{remark}

\subsection{Elementary nonlinear bounds}
We first record the operator bounds used throughout.
\begin{lemma}\label{lem:wp_operator_bounds}
There exist constants
\(C_{\mathcal N},\ C_{\mathcal M,1},\ C_{\mathcal M,\infty},\ C_{\widehat B},\ C_\kappa>0\)
such that, for all \(x,\widetilde x\in X\cap L^\infty(\Omega_l\times\Omega_c)^S\),
all \(s\in\mathcal S\), and a.e.\ \((l,c)\in\Omega_l\times\Omega_c\),
\begin{align}
|\mathcal N_s[x]| &\le C_{\mathcal N}\|x\|_X, \label{eq:wp_N_bound}\\
|\mathcal N_s[x]-\mathcal N_s[\widetilde x]|
&\le C_{\mathcal N}\|x-\widetilde x\|_X, \label{eq:wp_N_Lip}\\
\|\mathcal M_s^{\mathrm{in}}[x]\|_{L^1}
&\le C_{\mathcal M,1}\|x\|_X, \label{eq:wp_Min_L1}\\
\|\mathcal M_s^{\mathrm{in}}[x]-\mathcal M_s^{\mathrm{in}}[\widetilde x]\|_{L^1}
&\le C_{\mathcal M,1}\|x-\widetilde x\|_X, \label{eq:wp_Min_L1_Lip}\\
\|\mathcal M_s^{\mathrm{in}}[x]\|_{L^\infty}
&\le C_{\mathcal M,\infty}\|x\|_{L^\infty(\Omega_l\times\Omega_c)^S},
\label{eq:wp_Min_Linfty}\\
0\le \widehat B_s[x] &\le C_{\widehat B}\|x\|_X, \label{eq:wp_Bhat_bound}\\
|\widehat B_s[x]-\widehat B_s[\widetilde x]|
&\le C_{\widehat B}\|x-\widetilde x\|_X, \label{eq:wp_Bhat_Lip}\\
0\le \int_{\Omega_l}\!\int_{\Omega_c}\kappa_s x_s\,dc\,dl
&\le C_\kappa \|x_s\|_{L^1}.
\label{eq:wp_kappa_bound}
\end{align}
\end{lemma}
\begin{proof}
Results follow from Hölder's inequality and the essential boundedness of
\(\chi_{s,j,r}\), \(\theta_{rs}\), \(\varpi_s\), and \(\kappa_s\).
\end{proof}

\begin{lemma}\label{lem:wp_birth_bounds}
Fix \(M>0\) and set \(K_Y:=Y([0,T])\subset\mathbb R^{d_Y}\). There exists
\(C_{\mathcal B}(M,K_Y)>0\) such that if
\(\|x\|_X, \|\widetilde x\|_X, R, \widetilde R\in[0,M]\)
and \(y,\widetilde y\in K_Y\),
then, for every \(s\in\mathcal S\),
\begin{align}
\|\mathcal B_s[x,R,y]\|_{L^1(\Omega_c)}
&\le C_{\mathcal B}\|x_s\|_{L^1(\Omega_l\times\Omega_c)},
\label{eq:wp_B_L1}\\
\|\mathcal B_s[x,R,y]\|_{L^\infty(\Omega_c)}
&\le C_{\mathcal B}\|x_s\|_{L^1(\Omega_l\times\Omega_c)},
\label{eq:wp_B_Linfty}\\
\|\mathcal B_s[x,R,y]-\mathcal B_s[\widetilde x,\widetilde R,\widetilde y]\|_{L^1(\Omega_c)}
&\le
C_{\mathcal B}
\big(\|x-\widetilde x\|_X+|R-\widetilde R|+|y-\widetilde y|\big).
\label{eq:wp_B_Lip}
\end{align}
\end{lemma}
\begin{proof}
By Assumption~\ref{ass:model_birth_operator}, \(\beta_s\) is locally bounded and
\(\Pi_s\) is locally essentially bounded in \((R,y)\); also
\(\int_{\Omega_c}\Pi_s\,dc=1\).
The Lipschitz estimate follows by adding and subtracting intermediate terms.
\end{proof}

\begin{lemma}\label{lem:wp_coeff_bounds}
Fix \(M>0\). There exists \(C_M>0\) such that if
\[
\sup_{t}\|x(t)\|_X,\quad\sup_t\|\widetilde x(t)\|_X,\quad\sup_t|R(t)|,\quad\sup_t|\widetilde R(t)|\le M,
\]
then, for all \(t\in[0,T]\), \(s\in\mathcal S\),
the frozen coefficients \(g_s,h_s,\mu_s\) are uniformly bounded by \(C_M\) and are
Lipschitz in \((\|x(t)-\widetilde x(t)\|_X+|R_s(t)-\widetilde R_s(t)|)\) with constant
\(C_M\).
\end{lemma}

\begin{proof}[Sketch of proof]
Fix \(M>0\). By Lemma~\ref{lem:wp_operator_bounds}, for every \(s\in\mathcal S\) and
every \(t\in[0,T]\),
\[
|\mathcal N_s[x(t)]|
\le C_{\mathcal N}\|x(t)\|_X
\le C_{\mathcal N}M,
\qquad
|\mathcal N_s[\widetilde x(t)]|
\le C_{\mathcal N}M .
\]
Moreover \(0\le R_s(t),\widetilde R_s(t)\le M\), and the environmental path satisfies
\(Y(t)\in K_Y:=Y([0,T])\), where \(K_Y\subset\mathbb R^{d_Y}\) is compact. Hence all
arguments of \(g_s,h_s,\mu_s\) lie in the compact set
\[
\mathcal K_M
:=
[0,C_{\mathcal N}M]^m
\times[0,M]
\times K_Y
\times\Omega_l
\times\Omega_c .
\]
Since \(g_s,h_s,\mu_s\in C^1\) and their first-order derivatives are locally bounded,
there exists a constant \(C_M>0\), depending only on \(M\), \(K_Y\), and the standing
structural constants, such that on \(\mathcal K_M\)
\[
|g_s|+|h_s|+|\mu_s|
+
|\nabla_{(n,R)}g_s|
+
|\nabla_{(n,R)}h_s|
+
|\nabla_{(n,R)}\mu_s|
\le C_M .
\]
This gives the asserted uniform boundedness of the frozen coefficients.

For the Lipschitz estimate, apply the mean-value theorem in the variables
\((n,R)\), with \(y,l,c\) fixed. For example,
\[
\begin{aligned}
&\big|
g_s(\mathcal N_s[x(t)],R_s(t),Y(t),l,c)
-
g_s(\mathcal N_s[\widetilde x(t)],\widetilde R_s(t),Y(t),l,c)
\big|
\\
&\qquad
\le
C_M
\Big(
|\mathcal N_s[x(t)]-\mathcal N_s[\widetilde x(t)]|
+
|R_s(t)-\widetilde R_s(t)|
\Big).
\end{aligned}
\]
Using the bounded-linearity estimate
\[
|\mathcal N_s[x(t)]-\mathcal N_s[\widetilde x(t)]|
\le C_{\mathcal N}\|x(t)-\widetilde x(t)\|_X,
\]
we obtain
\[
\begin{aligned}
&\big\|
g_s(\mathcal N_s[x(t)],R_s(t),Y(t),\cdot,\cdot)
-
g_s(\mathcal N_s[\widetilde x(t)],\widetilde R_s(t),Y(t),\cdot,\cdot)
\big\|_{L^\infty}
\\
&\qquad
\le
C_M
\big(
\|x(t)-\widetilde x(t)\|_X
+
|R_s(t)-\widetilde R_s(t)|
\big),
\end{aligned}
\]
after increasing \(C_M\) if necessary. The same argument applies to \(h_s\) and
\(\mu_s\). This proves the claimed coefficient bounds and Lipschitz dependence.
\end{proof}

\subsection{Frozen linear problem}
Fix
\(\bar x\in C([0,T];X)\cap L^\infty(Q_T)^S\) and
\(\bar R\in C([0,T];\mathbb R_+^S)\).
For each \(s\in\mathcal S\), define the frozen coefficients
\begin{align}
a_s &:=g_s(\mathcal N_s[\bar x],\bar R_s,Y,l,c), &
b_s &:=h_s(\mathcal N_s[\bar x],\bar R_s,Y,l,c), \label{eq:wp_frozen_ab}\\
d_s &:=\mu_s(\mathcal N_s[\bar x],\bar R_s,Y,l,c)
+u_s^{(q)}[\bar x,Y]+\theta_s^{\mathrm{out}}, &
f_s &:=\mathcal M_s^{\mathrm{in}}[\bar x], \label{eq:wp_frozen_df}\\
\zeta_s &:=\mathcal B_s[\bar x,\bar R_s,Y]. \label{eq:wp_frozen_zeta}
\end{align}
The frozen linear problem is
\begin{equation}\label{eq:wp_frozen_problem}
\begin{cases}
\partial_t z_s+\partial_l(a_s z_s)+\partial_c(b_s z_s)
=-d_s z_s+f_s,
& (t,l,c)\in Q_T,\\[1mm]
a_s(t,l_0,c)\,z_s(t,l_0,c)=\zeta_s(t,c),
& (t,c)\in (0,T)\times\Omega_c,\\[1mm]
z_s(0,l,c)=\phi_s(l,c),
& (l,c)\in\Omega_l\times\Omega_c.
\end{cases}
\end{equation}
The linear theory (cf.\ \cite{yu2026age,Perthame2007,engel2000one,cai2026optimal,Webb1985,magal2018theory,wang2025multi}) is proved in Appendix~\ref{app:linear_transport}; the
stability step uses characteristic comparison estimates.

\begin{proposition}[Frozen linear well-posedness and estimates]\label{prop:wp_linear_package}
For every frozen pair \((\bar x,\bar R)\), problem
\eqref{eq:wp_frozen_problem} has a unique weak solution
\(z_s\in C([0,T];L^1(\Omega_l\times\Omega_c))\cap L^\infty(Q_T)\),
\(z_s\ge 0\).
Moreover, there exist constants \(C_{\mathrm{div}},C_{\mathrm{lin}}(T)>0\), depending
only on the structural bounds, such that for all \(t\in[0,T]\),
\begin{align}
\|z_s(t)\|_{L^1}
&\le
e^{C_{\mathrm{div}}t}\Big(
\|\phi_s\|_{L^1}
+\int_0^t e^{-C_{\mathrm{div}}\tau}\big[\|\zeta_s(\tau)\|_{L^1(\Omega_c)}
+\|f_s(\tau)\|_{L^1}\big]\,d\tau
\Big),
\label{eq:wp_linear_L1}
\\
\|z_s\|_{L^\infty(Q_t)}
&\le
C_{\mathrm{lin}}(T)\Big(
\|\phi_s\|_{L^\infty}
+\|\zeta_s\|_{L^\infty((0,t)\times\Omega_c)}
+\int_0^t\|f_s(\tau)\|_{L^\infty}\,d\tau
\Big).
\label{eq:wp_linear_Linfty}
\end{align}
\end{proposition}
\begin{proof}
We first show well-posedness. Approximate by smooth data. The characteristic representation provides uniform
bounds (below). Passage to the limit uses \(L^1\)-compactness; uniqueness uses the
\(L^1\)-estimate applied to the zero-data problem.

We show the $L^1$ estimate \eqref{eq:wp_linear_L1}.
For a smooth nonnegative solution, integrate \eqref{eq:wp_frozen_problem} over
\(\Omega_l\times\Omega_c\). The boundary contribution at \(l=l_0\) yields the
inflow \(\int_{\Omega_c}\zeta_s\,dc\). The contribution at \(l=l_m\) is
\(-\int_{\Omega_c}a_s(t,l_m,c)z_s(t,l_m,c)\,dc\le 0\). The \(c\)-boundary
contribution is nonpositive by Assumption~\ref{ass:model_condition_boundary_orientation}.

For the interior terms, rewrite the transport in non-conservative form:
\[
\partial_l(a_s z_s)+\partial_c(b_s z_s)
=a_s\partial_l z_s+b_s\partial_c z_s+(\partial_l a_s+\partial_c b_s)z_s.
\]
After integration by parts, the divergence contribution
\(\int(\partial_l a_s+\partial_c b_s)z_s\,dc\,dl\)
need not have a definite sign. Combined with \(-d_s z_s\), the total zeroth-order
contribution satisfies
\[
\int(\partial_l a_s+\partial_c b_s-d_s)z_s\,dc\,dl
\le C_{\mathrm{div}}\|z_s(t)\|_{L^1},
\]
since \((\partial_l a_s+\partial_c b_s-d_s)\le(\partial_l a_s+\partial_c b_s)^-\le C_{\mathrm{div}}\)
and \(z_s\ge 0\). Hence
\[
\frac{d}{dt}\|z_s(t)\|_{L^1}
\le C_{\mathrm{div}}\|z_s(t)\|_{L^1}
+\|\zeta_s(t)\|_{L^1(\Omega_c)}+\|f_s(t)\|_{L^1}.
\]
The Gr\"onwall lemma yields \eqref{eq:wp_linear_L1}. For signed solutions, the argument extends
by decomposing \(z_s=z_s^+-z_s^-\) and estimating each part, giving the same exponential
bound.

We show the $L^\infty$ estimate \eqref{eq:wp_linear_Linfty}.
The characteristic representation
\eqref{eq:appA_rep_init}--\eqref{eq:appA_rep_boundary} gives pointwise bounds.
Since \(E_s(\tau,t)\le e^{C_{\mathrm{div}}(t-\tau)}\) and
\(1/a_s(\tau^{\mathrm{ent}},l_0,\cdot)\le 1/g_{\min}\), the estimate follows.
\end{proof}

\begin{proposition}[Frozen linear \(L^1\) stability]\label{prop:wp_linear_stability}
Let \(z^{(1)}\), \(z^{(2)}\) solve two frozen problems with the same initial data
\(\phi\), corresponding to frozen pairs
\((\bar x^{(i)},\bar R^{(i)})\), \(i=1,2\), with
\(\|(x^{(i)},R^{(i)})\|_{\mathbb F_T}\le M\) and
\(\|\bar x^{(i)}\|_{L^\infty(Q_T)^S}\le M\).
Set \(w:=z^{(1)}-z^{(2)}\). Then, for all \(t\in[0,T]\),
\begin{align}\label{eq:wp_linear_stability}
\begin{aligned}
\|w(t)\|_X^2
&\le
C_{\mathrm{stab}}(M,T)
\int_0^t \|w(\tau)\|_X^2\,d\tau \\[4pt]
&+
C_{\mathrm{stab}}(M,T)
\int_0^t
\Big(\|\bar x^{(1)}(\tau)-\bar x^{(2)}(\tau)\|_X
+|\bar R^{(1)}(\tau)-\bar R^{(2)}(\tau)|\Big)^2d\tau.
\end{aligned}
\end{align}
\end{proposition}
\begin{proof}
Apply the duality stability estimate (Proposition~\ref{prop:appA_stability}) zone by zone,
noting that the coefficient differences \(\delta a_s,\delta b_s,\delta d_s\) and
the boundary/source differences \(\delta\zeta_s,\delta f_s\) are all controlled by
\(\|\bar x^{(1)}-\bar x^{(2)}\|_X+|\bar R^{(1)}-\bar R^{(2)}|\)
via Lemmas~\ref{lem:wp_birth_bounds} and~\ref{lem:wp_coeff_bounds}.
\end{proof}

\subsection{The fixed-point map}
For \(T_0\in(0,T]\) and \(M>0\), let
\[
\mathfrak B_M^{T_0}
:=
\Big\{
(x,R)\in\mathbb F_{T_0}:
x(t)\in X_+,\ R(t)\in\mathbb R_+^S,\
\|(x,R)\|_{\mathbb F_{T_0}}\le M
\Big\}.
\]
Given \((\bar x,\bar R)\in\mathfrak B_M^{T_0}\), define
\(\Gamma(\bar x,\bar R)=(z,\widetilde R)\)
where: (1)~for each \(s\), \(z_s\) solves
\eqref{eq:wp_frozen_problem} with frozen coefficients from \((\bar x,\bar R)\);
(2)~\(\widetilde R\) solves
\begin{equation}\label{eq:wp_frozen_resource}
\widetilde R_s(t)
=R_{s,0}
+\int_0^t
\Big[
F_s(\widetilde R_s(\tau),Y(\tau))
-\rho_s(\widetilde R_s(\tau))
\int\kappa_s\bar x_s(\tau)\,dc\,dl
\Big]d\tau.
\end{equation}

\begin{lemma}[Resource bounds]\label{lem:wp_resource_bounds}
For every \(\bar x\in C([0,T_0];X_+)\) with
\(\sup_t\|\bar x(t)\|_X\le M\),
eq.~\eqref{eq:wp_frozen_resource} has a unique
\(\widetilde R\in W^{1,\infty}(0,T_0;\mathbb R^S)\cap C([0,T_0];\mathbb R_+^S)\)
with \(0\le \widetilde R_s(t)\le C_R(M,T_0)\).
\end{lemma}
\begin{proof}
Existence and uniqueness follow from the local Lipschitz continuity of \(F_s,\rho_s\).
The bound \(\widetilde R_s\ge 0\) uses the Nagumo argument of
Remark~\ref{rem:resource_invariance} (with \(\bar x_s\ge 0\)). The upper bound uses
\eqref{eq:wp_kappa_bound} and the local boundedness of \(F_s\).
\end{proof}

\begin{lemma}[Self-map]\label{lem:wp_selfmap}
For every \(M\ge 2(\|\phi\|_X+|R_0|+1)\), there exists \(T_M\in(0,T]\) such that
\(\Gamma(\mathfrak B_M^{T_M})\subset \mathfrak B_M^{T_M}\).
\end{lemma}
\begin{proof}
By \eqref{eq:wp_Min_L1}, \eqref{eq:wp_B_L1}, and \eqref{eq:wp_linear_L1},
\[
\sup_{t\in[0,T_M]}\|z(t)\|_X
\le
C(M,T_M)\big(\|\phi\|_X+T_M M\big).
\]
By Lemma~\ref{lem:wp_resource_bounds},
\(\sup_t|\widetilde R(t)|\le |R_0|+C(M)T_M\).
Choosing \(T_M\) small gives
\(\|(z,\widetilde R)\|_{\mathbb F_{T_M}}\le M\).
\end{proof}

\begin{lemma}[Contraction]\label{lem:wp_contraction}
For every \(M>0\), there exists \(\tau_M\in(0,T_M]\) such that
\(\Gamma:\mathfrak B_M^{\tau_M}\to\mathfrak B_M^{\tau_M}\)
is a strict contraction in \(\|\cdot\|_{\mathbb F_{\tau_M}}\).
\end{lemma}
\begin{proof}
Let \(\Gamma(\bar x^{(i)},\bar R^{(i)})=(z^{(i)},\widetilde R^{(i)})\), \(i=1,2\).
Set \(D(t):=\|z^{(1)}(t)-z^{(2)}(t)\|_X^2+|\widetilde R^{(1)}(t)-\widetilde R^{(2)}(t)|^2\)
and \(\bar D(t):=\|\bar x^{(1)}(t)-\bar x^{(2)}(t)\|_X^2+|\bar R^{(1)}(t)-\bar R^{(2)}(t)|^2\).
By Proposition~\ref{prop:wp_linear_stability} and the resource Lipschitz estimate:
\(D(t)\le C(M)\int_0^t D(\tau)\,d\tau+C(M)\int_0^t \bar D(\tau)\,d\tau\).
Since \(z^{(i)},\widetilde R^{(i)}\) are the outputs of \(\Gamma\) (not the
iterates themselves), and the first integral involves \(D\) (not \(\bar D\)):
Gr\"onwall gives
\(D(t)\le C(M)e^{C(M)\tau_M}\int_0^t\bar D(\tau)\,d\tau
\le C(M)\tau_M e^{C(M)\tau_M}\sup_\tau \bar D(\tau)\).
Taking supremum over \(t\in[0,\tau_M]\) and square roots:
\(\sup_t\big(\|z^{(1)}-z^{(2)}\|_X+|\widetilde R^{(1)}-\widetilde R^{(2)}|\big)
\le C'(M)\sqrt{\tau_M}\,
\sup_t\big(\|\bar x^{(1)}-\bar x^{(2)}\|_X+|\bar R^{(1)}-\bar R^{(2)}|\big)\).
Choose \(\tau_M\) so that \(C'(M)\sqrt{\tau_M}<1\).
\end{proof}

\subsection{Local well-posedness}
\begin{theorem}[Local well-posedness]\label{thm:wp_local}
There exists \(\tau>0\) such that the coupled system admits a unique weak solution
\[
(x,R)\in
\big(C([0,\tau];X)\cap L^\infty(Q_\tau)^S\big)
\times
\big(W^{1,\infty}(0,\tau;\mathbb R^S)\cap C([0,\tau];\mathbb R_+^S)\big),
\]
with \(x(t)\in X_+\) and \(R(t)\in\mathbb R_+^S\) for all \(t\in[0,\tau]\).
\end{theorem}
\begin{proof}
\emph{Existence.}
By Lemmas~\ref{lem:wp_selfmap} and~\ref{lem:wp_contraction}, the Picard iterates
\((x^{(n)},R^{(n)})=\Gamma^n(x^{(0)},R^{(0)})\)
are Cauchy in \(\|\cdot\|_{\mathbb F_\tau}\), hence converge to a fixed point
\((x,R)\in\mathfrak B_M^\tau\).
The \(L^\infty\) bound follows from Lemma~\ref{lem:wp_Linfty} applied to the limit.

\emph{Uniqueness.}
Suppose \((x_1,R_1)\) and \((x_2,R_2)\) are two weak solutions on \([0,\tau]\) with the
same data. Set \(w_s:=x_{1,s}-x_{2,s}\). The coefficient differences
\(\delta g_s,\delta h_s,\delta\mu_s\) satisfy
\(\|\delta g_s\|_\infty+\|\delta h_s\|_\infty+\|\delta\mu_s\|_\infty
\le C_M(\|w\|_X+|\delta R|)\),
and similarly the boundary and source differences are controlled by
\(\|w\|_X+|\delta R|\).
By the duality estimate \eqref{eq:appA_sqrt_stability} (applied zone by zone),
\[
\|w(t)\|_X
\le
2\sqrt{C'\int_0^t(\|w(\sigma)\|_X+|\delta R(\sigma)|)\,d\sigma}
+C\int_0^t(\|w(\sigma)\|_X+|\delta R(\sigma)|)\,d\sigma.
\]
The resource difference satisfies
\(|\delta R(t)|\le C\int_0^t(\|w(\sigma)\|_X+|\delta R(\sigma)|)\,d\sigma\) by Lipschitz
continuity. Set \(\Phi(t):=\|w(t)\|_X^2+|\delta R(t)|^2\). Squaring the above and
using the \(|\delta R|\) estimate gives, after straightforward manipulations,
\[
\Phi(t)\le C\int_0^t\Phi(\sigma)\,d\sigma.
\]
Gr\"onwall's inequality and \(\Phi(0)=0\) yield \(\Phi\equiv 0\),
hence \(w\equiv 0\) and \(\delta R\equiv 0\).
\end{proof}

\subsection{A priori bounds and global continuation}
\begin{lemma}[\(L^1\)-bound]\label{lem:wp_L1}
Let \((x,R)\) be a nonnegative weak solution on \([0,\tau]\). Then there exists
\(C_1>0\), depending only on the structural data, \(Y([0,\tau])\), and \(\tau\), such
that
\begin{equation}\label{eq:wp_apriori_L1}
\sup_{t\in[0,\tau]}\|x(t)\|_X
\le
C_1\bigl(1+\|\phi\|_X\bigr)e^{C_1\tau}.
\end{equation}
\end{lemma}
\begin{proof}
Apply \eqref{eq:wp_linear_L1} to the fixed point.
Using \eqref{eq:wp_Min_L1} and~\eqref{eq:wp_B_L1}, sum over \(s\):
\[
\|x(t)\|_X
\le
e^{C_{\mathrm{div}}t}\bigg(
\|\phi\|_X
+C\int_0^t e^{-C_{\mathrm{div}}\sigma}(1+\|x(\sigma)\|_X)\,d\sigma\bigg).
\]
The Gr\"onwall lemma yields \eqref{eq:wp_apriori_L1}.
\end{proof}

\begin{lemma}[Resource bound]\label{lem:wp_Rbound}
Under the same hypotheses, \(0\le R_s(t)\le C_2\) for all \(t\in[0,\tau]\).
\end{lemma}

\begin{lemma}[\(L^\infty\)-bound]\label{lem:wp_Linfty}
Under the same hypotheses, there exists \(C_\infty>0\) depending on the structural data,
\(Y([0,\tau])\), \(\tau\), and the \(L^1\)- and resource bounds, such that
\begin{equation}\label{eq:wp_apriori_Linfty}
\|x\|_{L^\infty(Q_\tau)^S}\le C_\infty.
\end{equation}
\end{lemma}
\begin{proof}
Apply \eqref{eq:wp_linear_Linfty} to the fixed point.
By \eqref{eq:wp_B_Linfty},
\(\|\mathcal B_s[x,R_s,Y]\|_{L^\infty(\Omega_c)}
\le C_{\mathcal B}\sup_\sigma\|x_s(\sigma)\|_{L^1}\),
and by \eqref{eq:wp_Min_Linfty},
\(\|\mathcal M_s^{\mathrm{in}}[x(\sigma)]\|_{L^\infty}
\le C_{\mathcal M,\infty}\|x(\sigma)\|_{(L^\infty)^S}\).
Hence
\[
\|x_s\|_{L^\infty(Q_t)}
\le
C_{\mathrm{lin}}(\tau)\Big(
\|\phi_s\|_{L^\infty}
+C_{\mathcal B}\sup_\sigma\|x_s(\sigma)\|_{L^1}
+C_{\mathcal M,\infty}\int_0^t\|x(\sigma)\|_{(L^\infty)^S}\,d\sigma\Big).
\]
The first two terms are finite by \eqref{eq:wp_apriori_L1}. Summing over \(s\) and
applying Gr\"onwall to the \(L^\infty\) integral term yields
\eqref{eq:wp_apriori_Linfty}.
\end{proof}

\begin{theorem}[Global well-posedness]\label{thm:wp_global}
The coupled system admits a unique weak solution on \([0,T]\):
\[
(x,R)\in
\big(C([0,T];X)\cap L^\infty(Q_T)^S\big)
\times
\big(W^{1,\infty}(0,T;\mathbb R^S)\cap C([0,T];\mathbb R_+^S)\big).
\]
\end{theorem}
\begin{proof}
Let \([0,T_{\max})\) be the maximal existence interval from Theorem~\ref{thm:wp_local}.
By the standard continuation principle, if \(T_{\max}<T\) then necessarily
\begin{equation}\label{eq:wp_blowup_alternative}
\limsup_{t\uparrow T_{\max}}
\Big(\|x(t)\|_X+\|x\|_{L^\infty(Q_t)^S}+|R(t)|\Big)=\infty.
\end{equation}
But Lemmas~\ref{lem:wp_L1}--\ref{lem:wp_Linfty} provide a priori bounds depending only
on the data and elapsed time, ruling out \eqref{eq:wp_blowup_alternative}.
Hence \(T_{\max}=T\). Uniqueness propagates by patching local uniqueness intervals.
\end{proof}

\subsection{Continuous dependence}
\begin{theorem}[Continuous dependence]\label{thm:wp_continuous_dependence}
Let
\((\phi^{(n)},R_0^{(n)},Y^{(n)},q^{(n)})
\to(\phi,R_0,Y,q)\)
\emph{strongly} in
\(X\times\mathbb R^S\times C([0,T];\mathbb R^{d_Y})\times L^\infty(0,T;\mathbb R^S)\).
Then
\begin{equation}\label{eq:wp_cont_dep}
\sup_{t\in[0,T]}\|x^{(n)}(t)-x(t)\|_X
+\sup_{t\in[0,T]}|R^{(n)}(t)-R(t)|
\longrightarrow 0.
\end{equation}
If, in addition, \(\|\phi^{(n)}-\phi\|_{(L^\infty)^S}\to 0\), then also
\(\|x^{(n)}-x\|_{L^\infty(Q_T)^S}\to 0\).
\end{theorem}
\begin{proof}
All solutions lie in a common bounded ball by uniform a priori bounds. On a short
interval \([0,\tau]\), the coefficient differences for the \(n\)-th and limiting
solutions arise from two sources. First, from the data difference:
\begin{align*}
&\|g_s(\mathcal N_s[x^{(n)}],R_s^{(n)},Y^{(n)},\cdot)
-g_s(\mathcal N_s[x],R_s,Y,\cdot)\|_{L^\infty}\\
&\quad\le
C_M\big(\|x^{(n)}-x\|_X+|R^{(n)}-R|+\|Y^{(n)}-Y\|_C\big),
\end{align*}
and similarly for \(h_s,\mu_s\). Second, from the control change:
\[
\|u_s^{(q^{(n)})}-u_s^{(q)}\|_{L^\infty}
\le C_M\big(\|q^{(n)}-q\|_{L^\infty}+\|x^{(n)}-x\|_X\big).
\]
The boundary inflow \(\delta\zeta_s=\mathcal B_s[x^{(n)},R^{(n)},Y^{(n)}]
-\mathcal B_s[x,R,Y]\) satisfies
\(\|\delta\zeta_s\|_{L^1}\le C_M(\|x^{(n)}-x\|_X+|R^{(n)}-R|+\|Y^{(n)}-Y\|_C)\).
Inserting these into the duality stability estimate \eqref{eq:appA_sqrt_stability},
squaring, and using the Gr\"onwall structure (as in Theorem~\ref{thm:wp_local}), one obtains
\eqref{eq:wp_cont_dep} on \([0,\tau]\). For the \(L^\infty\) part, use
Proposition~\ref{prop:appA_Linfty_stability}, which requires the additional hypothesis
\(\|\phi^{(n)}-\phi\|_{L^\infty}\to 0\). Iterate over finitely many subintervals.
\end{proof}

\section{Reduced autonomous next-generation operator and threshold consequences}%
\label{sec:spectral_R0}

This section does not derive the generator framework from the PDE anew. Instead, it constructs the next-generation operator, records its spectral properties, and states the threshold consequence after the reduced linearized problem is embedded into the standard resolvent-positive semigroup setting \cite{Thieme2009,wang2026algebraic,huo2025growth}.

We freeze the environment, the control, and remove spatial coupling:
\begin{equation}\label{eq:r0_reduction}
S=1,\qquad Y(t)\equiv y^\ast\in\mathbb R^{d_Y},\qquad
q(t)\equiv \bar q\in[0,q_{\max}],\qquad \theta_{rs}\equiv 0.
\end{equation}

\subsection{Extinction equilibrium}
At \(x\equiv 0\), \(\mathcal N[0]=0\), \(\widehat B[0]=0\), and the resource
satisfies \(F(R^\circ,y^\ast)=0\).
\begin{assumption}[Extinction resource equilibrium]\label{ass:r0_resource_eq}
Equation \(F(R^\circ,y^\ast)=0\) has a unique solution
\(R^\circ=R^\circ(y^\ast)\in\mathbb R_+\), and
\(\partial_R F(R^\circ,y^\ast)<0\).
\end{assumption}

\subsection{Linearization at extinction}
Freeze all coefficients at \((0,R^\circ)\) and define
\begin{align}
g^\circ(l,c)&:=g(0,R^\circ,y^\ast,l,c), &
h^\circ(l,c)&:=h(0,R^\circ,y^\ast,l,c), \label{eq:r0_gh0}\\
\mu^\circ(l,c)&:=\mu(0,R^\circ,y^\ast,l,c), &
u^\circ(l,c)&:=\bar q\,\sigma(l,c)\eta(0,y^\ast), \label{eq:r0_mu_u0}\\
\beta^\circ(l,c')&:=\beta(0,R^\circ,y^\ast,l,c'), &
\Pi^\circ(c\,|\,l,c')&:=\Pi(c\,|\,l,c',R^\circ,y^\ast). \label{eq:r0_beta_Pi0}
\end{align}

\subsection{Characteristic representation}
The autonomous characteristic system
\begin{equation}\label{eq:r0_chars}
\dot L(\tau)=g^\circ(L,C),\quad
\dot C(\tau)=h^\circ(L,C),\quad
L(0)=l_0,\ C(0)=c_0,
\end{equation}
has a unique maximal solution for each \(c_0\in\Omega_c\),
with \(L(\tau;c_0)\) strictly increasing. Define
\[
\tau_{\max}(c_0)
:=\sup\{\tau>0:(L(s;c_0),C(s;c_0))\in(l_0,l_m)\times(0,c_m)\ \forall s\in[0,\tau)\}.
\]
Then
\begin{equation}\label{eq:r0_tau_bound}
0<\tau_{\max}(c_0)\le \overline\tau_{\max}:=\frac{l_m-l_0}{g_{\min}}
\qquad\forall\,c_0\in\Omega_c.
\end{equation}
Define the survival factor
\begin{equation}\label{eq:r0_survival}
\ell(\tau;c_0)
:=
\exp\!\Big(
-\int_0^\tau
\big[
\partial_l g^\circ+\partial_c h^\circ+\mu^\circ+u^\circ
\big](L(s;c_0),C(s;c_0))\,ds
\Big).
\end{equation}

\subsection{Next-generation operator}
Let \(\mathcal V:=L^1(\Omega_c)\), \(\mathcal V_+:=L^1_+(\Omega_c)\).
Define the kernel
\begin{equation}\label{eq:r0_kernel}
\mathscr K(c,c_0;y^\ast,\bar q)
:=
\int_0^{\tau_{\max}(c_0)}
\beta^\circ(L(\tau;c_0),C(\tau;c_0))\,
\Pi^\circ(c\,|\,L(\tau;c_0),C(\tau;c_0))\,
\ell(\tau;c_0)\,d\tau,
\end{equation}
and the integral operator
\begin{equation}\label{eq:r0_operator}
(\mathcal K(y^\ast,\bar q)\psi)(c)
:=
\int_{\Omega_c}\mathscr K(c,c_0;y^\ast,\bar q)\,\psi(c_0)\,dc_0.
\end{equation}

\begin{assumption}[Kernel regularity]\label{ass:r0_kernel}
For each fixed \((y^\ast,\bar q)\):
\begin{enumerate}[label=\textup{(K\arabic*)}]
\item The integrand in \eqref{eq:r0_kernel} is measurable and dominated by an integrable
majorant independent of \((c,c_0)\).
\item \(\mathscr K(\cdot,\cdot;y^\ast,\bar q)\in C(\Omega_c\times\Omega_c)\).
\end{enumerate}
\end{assumption}

\begin{proposition}\label{prop:r0_compact}
Under Assumptions~\ref{ass:model_vital_rates}, \ref{ass:model_birth_operator}, \ref{ass:r0_resource_eq} and~\ref{ass:r0_kernel},
\(\mathcal K(y^\ast,\bar q):\mathcal V\to\mathcal V\) is bounded, positive, and compact.
More precisely, \(\mathcal K\) maps \(L^1(\Omega_c)\) into \(C(\Omega_c)\).
\end{proposition}
\begin{proof}
Positivity is immediate. Boundedness follows from
\(\sup_{c_0}\int_{\Omega_c}\mathscr K\,dc\le C_{\mathscr K}\).
Since \(\mathscr K\in C(\Omega_c\times\Omega_c)\) and \(\Omega_c=[0,c_m]\) is compact,
\(\mathscr K\) is uniformly continuous. For any \(\psi\in L^1(\Omega_c)\),
the function \(c\mapsto(\mathcal K\psi)(c)=\int\mathscr K(c,c_0)\psi(c_0)\,dc_0\) is
continuous, so \(\mathcal K:L^1(\Omega_c)\to C(\Omega_c)\).
The image of the unit ball of \(L^1(\Omega_c)\) under \(\mathcal K\) is uniformly bounded
and equicontinuous in \(C(\Omega_c)\) (by uniform continuity of \(\mathscr K\)).
Arzel\`a--Ascoli gives relative compactness in \(C(\Omega_c)\),
hence also in \(L^1(\Omega_c)\).
\end{proof}

\begin{definition}[Reduced basic reproduction number]\label{def:r0}
\(\mathcal R_0(y^\ast,\bar q):=r(\mathcal K(y^\ast,\bar q))\).
\end{definition}

\begin{theorem}[Basic spectral properties]\label{thm:r0_basic}
The following hold.
\begin{enumerate}[label=\textup{(\alph*)}]
\item\label{item:r0_eig}
\(r(\mathcal K)\) is an eigenvalue with a nonnegative eigenfunction
\(\psi^\ast\in\mathcal V_+\setminus\{0\}\)
\textup{(Krein--Rutman \cite{KreinRutman1948,yu2026beyond,phat1994kreuin})}.
\item\label{item:r0_qmono}
\(\mathcal R_0(y^\ast,\bar q)\) is nonincreasing in \(\bar q\).
\item\label{item:r0_mono_coeff}
\(\mathcal R_0\) is pointwise nondecreasing in \(\beta^\circ\) and nonincreasing in
\(\mu^\circ\).
\end{enumerate}
\end{theorem}
\begin{proof}
Part~\ref{item:r0_eig}: Proposition~\ref{prop:r0_compact} and Krein--Rutman.
Parts~\ref{item:r0_qmono}--\ref{item:r0_mono_coeff}: increasing \(\bar q\) (resp.\
\(\mu^\circ\)) decreases \(\ell\) in \eqref{eq:r0_survival}, hence decreases \(\mathscr K\)
pointwise. Monotonicity of the spectral radius for positive operators gives the
conclusion.
\end{proof}

\subsection{Continuity with respect to the frozen environment}
\begin{assumption}[Parameter continuity]\label{ass:r0_param_cont}
Let \(\mathscr Y\subset\mathbb R^{d_Y}\) be compact. The maps
\(y^\ast\mapsto R^\circ(y^\ast)\),
\(y^\ast\mapsto g^\circ,h^\circ,\mu^\circ,\beta^\circ,\Pi^\circ,\eta(0,y^\ast)\)
are continuous in the natural sup-norm topologies on \(\mathscr Y\),
and the domination in Assumption~\ref{ass:r0_kernel} is uniform for
\(y^\ast\in\mathscr Y\).
\end{assumption}

\begin{theorem}[Continuity of \(\mathcal R_0\)]\label{thm:r0_cont}
Under Assumptions~\ref{ass:model_vital_rates}, \ref{ass:model_birth_operator}, \ref{ass:r0_resource_eq}, \ref{ass:r0_kernel}, and~\ref{ass:r0_param_cont}, for fixed \(\bar q\),
\(\mathscr Y\ni y^\ast\mapsto \mathcal R_0(y^\ast,\bar q)\)
is continuous.
\end{theorem}
\begin{proof}
Let \(y_n^\ast\to y^\ast\). Set
\(\hat{\mathcal K}_n:=\mathcal K(y_n^\ast,\bar q)\) and
\(\hat{\mathcal K}:=\mathcal K(y^\ast,\bar q)\).
By continuous dependence of ODE flows on parameters and Assumption~\ref{ass:r0_param_cont},
\(\|\hat{\mathcal K}_n-\hat{\mathcal K}\|_{\mathcal L(\mathcal V)}\to 0\).

\emph{Upper bound.} The spectral-radius formula
\(r(T)=\inf_{p\ge 1}\|T^p\|^{1/p}\)
gives: for each \(\varepsilon>0\), choose \(p\) with
\(\|\hat{\mathcal K}^p\|^{1/p}<r(\hat{\mathcal K})+\varepsilon\).
Since \(\hat{\mathcal K}_n^p\to\hat{\mathcal K}^p\) in norm,
\(\limsup r(\hat{\mathcal K}_n)\le r(\hat{\mathcal K})+\varepsilon\).
As \(\varepsilon\) is arbitrary, \(\limsup r(\hat{\mathcal K}_n)\le r(\hat{\mathcal K})\).

\emph{Lower bound.} If \(r(\hat{\mathcal K})=0\), we are done. Assume
\(r(\hat{\mathcal K})>0\). For any \(\varepsilon>0\), choose
\(\delta<\varepsilon\) such that the circle
\(\Gamma:=\{z:|z-r(\hat{\mathcal K})|=\delta\}\) avoids \(\sigma(\hat{\mathcal K})\).
The Riesz projection
\(P:=\frac{1}{2\pi i}\oint_\Gamma(zI-\hat{\mathcal K})^{-1}dz\neq 0\).
For \(n\) large, \(zI-\hat{\mathcal K}_n\) is invertible on \(\Gamma\) and
\(P_n\to P\); hence \(P_n\neq 0\) and \(\hat{\mathcal K}_n\) has an eigenvalue
\(\lambda_n\) with \(|\lambda_n-r(\hat{\mathcal K})|<\delta<\varepsilon\).
Therefore
\(r(\hat{\mathcal K}_n)\ge |\lambda_n|>r(\hat{\mathcal K})-\varepsilon\).
Since \(\varepsilon\) is arbitrary, \(\liminf r(\hat{\mathcal K}_n)\ge r(\hat{\mathcal K})\).
\end{proof}

\subsection{Threshold sign for the linearized autonomous problem}
To connect \(\mathcal R_0\) to the growth bound, we impose a semigroup hypothesis.
\begin{assumption}[Linear semigroup framework]\label{ass:r0_semigroup}
Let \(\mathfrak X:=L^1(\Omega_l\times\Omega_c)\). The linearized problem at
\((0,R^\circ)\) may be written
\(\dot z=(\mathcal A_0+\mathcal F)z\), where:
\begin{enumerate}[label=\textup{(S\arabic*)}]
\item \(\mathcal A_0\) is the transport--mortality operator with homogeneous inflow
and is resolvent-positive, with \(0\in\rho(\mathcal A_0)\).
\item \(\mathcal F\) is the positive boundary recruitment operator.
\item \(\mathcal A_0+\mathcal F\) generates a positive \(C_0\)-semigroup on
\(\mathfrak X\).
\item \(\mathcal K(y^\ast,\bar q)=-\mathcal F\mathcal A_0^{-1}\).
\end{enumerate}
\end{assumption}

\begin{remark}\label{rem:semigroup_conditional}
Assumption~\ref{ass:r0_semigroup} is \emph{not} derived from the PDE analysis in this paper. It
constitutes a hypothesis under which the standard resolvent-positive reproduction-number
theory applies \cite{Thieme2009,yu2026microscopic,AniKuzoBanasiak2019}. Verifying this hypothesis---constructing the domains, proving
generation, and confirming the identity (S4)---for the specific transport--renewal
operator at hand is a functional-analytic realization step that requires additional
work beyond the present scope.
\end{remark}

\begin{theorem}[Reduced threshold sign]\label{thm:r0_sign}
Under Assumptions~\ref{ass:model_vital_rates}, \ref{ass:model_birth_operator}, \ref{ass:r0_resource_eq}, \ref{ass:r0_kernel}, and~\ref{ass:r0_semigroup},
let \(s(\mathcal A_0+\mathcal F)\) denote the spectral bound. Then
\begin{equation}\label{eq:r0_sign}
s(\mathcal A_0+\mathcal F)\gtrless 0
\iff
\mathcal R_0(y^\ast,\bar q)\gtrless 1,
\end{equation}
and similarly at equality.
\end{theorem}
\begin{proof}
This is the infinite-dimensional reproduction-number theorem for resolvent-positive
operators, applied to \(\mathcal K=-\mathcal F\mathcal A_0^{-1}\).
\end{proof}

\subsection{Local nonlinear interpretation}
\begin{assumption}[Local nonlinear linearization]\label{ass:r0_local_nonlinear}
The reduced autonomous nonlinear problem is locally well posed near
\((0,R^\circ)\) and the associated semiflow is Fr\'echet differentiable there,
with linearization generated by \(\mathcal A_0+\mathcal F\).
\end{assumption}

\begin{proposition}[Local threshold]\label{prop:r0_local_threshold}
Assume Theorem~\ref{thm:r0_sign} and Assumption~\ref{ass:r0_local_nonlinear}. Then:
\begin{enumerate}[label=\textup{(\alph*)}]
\item \(\mathcal R_0<1\) implies local asymptotic stability of \((0,R^\circ)\).
\item \(\mathcal R_0>1\) implies linear instability of \((0,R^\circ)\).
\end{enumerate}
\end{proposition}
\begin{proof}
Standard principle of linearized stability for semiflows, using Theorem~\ref{thm:r0_sign} and Assumption~\ref{ass:r0_local_nonlinear}.
\end{proof}

\subsection{Scope}
The preceding results concern only the reduced autonomous problem at extinction.
No statement is made about:
global extinction when \(\mathcal R_0<1\);
uniform persistence when \(\mathcal R_0>1\);
a threshold theorem for the full nonautonomous multi-zone system; or
nonlinear asymptotics away from \((0,R^\circ)\).
The mathematically justified role of
\(\mathcal R_0(y^\ast,\bar q)=r(\mathcal K)\)
is:
\emph{local viability threshold for the reduced autonomous model at extinction,
conditional on the semigroup embedding hypothesis} (Assumption~\ref{ass:r0_semigroup}).

\section{Stationary states under a parametrized nonlinear renewal reduction}%
\label{sec:stationary}
We remain in the reduced autonomous regime \eqref{eq:r0_reduction}.
The goal is existence only; no uniqueness, bifurcation, or stability statement is made.

\subsection{Stationary problem}
A stationary pair \((x^\ast,R^\ast)\in X_+\times\mathbb R_+\) satisfies
\begin{equation}\label{eq:stationary_pde}
\partial_l(g\,x^\ast)+\partial_c(h\,x^\ast)
=-(\mu+\bar q\,\sigma\,\eta)x^\ast,
\end{equation}
with renewal boundary
\(g(\mathcal N[x^\ast],R^\ast,y^\ast,l_0,c)\,x^\ast(l_0,c)
=\mathcal B[x^\ast,R^\ast,y^\ast](c)\),
and resource closure
\(0=F(R^\ast,y^\ast)-\rho(R^\ast)\int\kappa\,x^\ast\,dc\,dl\).

\subsection{Parametrized operator hypothesis}
\begin{assumption}[Parametrized stationary renewal family]%
\label{ass:stationary_family}
There exists a family
\((\mathcal K_\alpha)_{\alpha\ge 0}\),
\(\mathcal K_\alpha:\mathcal V\to\mathcal V\)
bounded, positive, compact, such that:
\begin{enumerate}[label=\textup{(F\arabic*)}]
\item\label{ass:F_zero}
\(\mathcal K_0=\mathcal K(y^\ast,\bar q)\).
\item\label{ass:F_cont}
\(\alpha\mapsto \mathcal K_\alpha\) is continuous in operator norm.
\item\label{ass:F_mono}
\(0\le \alpha_1\le \alpha_2\) implies
\(\mathcal K_{\alpha_2}\psi\le \mathcal K_{\alpha_1}\psi\) for all
\(\psi\in\mathcal V_+\).
\item\label{ass:F_limit}
\(\limsup_{\alpha\to\infty}r(\mathcal K_\alpha)<1\).
\item\label{ass:F_rep}
If \(\mathcal K_\alpha\psi=\psi\) with
\(\psi\in\mathcal V_+\setminus\{0\}\),
then there exists a stationary weak solution
\((x^{(\alpha,\psi)},R^{(\alpha)})\)
whose inflow profile is \(\alpha\psi(c)\) and whose feedback
is compatible with the construction of \(\mathcal K_\alpha\).
\end{enumerate}
\end{assumption}

\begin{remark}[Interpretation of \(\alpha\)]
The scalar \(\alpha\) parametrizes the nonlinear feedback level. In the concrete class
verified in Appendix~\ref{app:stationary_family}, \(\alpha\) equals the
resource-weighted total abundance
\(\int\kappa x^\ast\,dc\,dl\), and \(\mathcal K_\alpha\) encodes the
density-dependent reduction in fecundity and increase in mortality.
\end{remark}

\begin{theorem}[Existence of a nontrivial stationary state]%
\label{thm:stationary_existence}
Assume Assumptions~\ref{ass:model_vital_rates}, \ref{ass:model_birth_operator}, \ref{ass:r0_resource_eq}, \ref{ass:r0_kernel}, and~\ref{ass:stationary_family}. If
\(\mathcal R_0(y^\ast,\bar q)=r(\mathcal K_0)>1\),
then the reduced stationary problem admits at least one nontrivial state
\((x^\ast,R^\ast)\in X_+\times\mathbb R_+\), \(x^\ast\not\equiv 0\).
\end{theorem}
\begin{proof}
By Assumption~\ref{ass:F_zero} and Proposition~\ref{prop:r0_compact}, \(r(\mathcal K_0)>1\).
By Assumption~\ref{ass:F_cont}, \(\alpha\mapsto r(\mathcal K_\alpha)\) is continuous (using
the spectral-radius continuity argument of Theorem~\ref{thm:r0_cont}).
By Assumption~\ref{ass:F_mono}, \(r(\mathcal K_\alpha)\) is nonincreasing.
By Assumption~\ref{ass:F_limit}, \(r(\mathcal K_\alpha)<1\) for large \(\alpha\).
The intermediate value theorem yields \(\alpha^\ast>0\) with
\(r(\mathcal K_{\alpha^\ast})=1\).
Krein--Rutman \cite{KreinRutman1948,liu2025bidirectional,Deimling1985} gives \(\psi^\ast\in\mathcal V_+\setminus\{0\}\) with
\(\mathcal K_{\alpha^\ast}\psi^\ast=\psi^\ast\).
By Assumption~\ref{ass:F_rep}, there exists a stationary solution with inflow
\(\alpha^\ast\psi^\ast\neq 0\), hence \(x^\ast\not\equiv 0\).
\end{proof}

This is purely existential. It does not imply uniqueness, bifurcation, stability, or
existence for the full multi-zone system.

\section{Finite-horizon optimal harvesting}\label{sec:optimal_harvest}
Fix \(T>0\) and \(Y\in C([0,T];\mathbb R^{d_Y})\). For each
\(q\in\mathcal Q_{\mathrm{ad}}^T\), let \((x^{(q)},R^{(q)})\) denote
the unique weak solution (Theorem~\ref{thm:wp_global}).

\subsection{Objective functional}
For \(s\in\mathcal S\), let
\(p_s\in C(\Omega_l\times\Omega_c\times\mathbb R^{d_Y})\),
\(0\le p_s\le p_{\max}\),
\(\Theta_s\in L^\infty_+(\Omega_l\times\Omega_c)\),
and \(\delta\ge 0\), \(\gamma\ge 0\). Define
\begin{equation}\label{eq:ocp_J}
\begin{aligned}
J(q)
&:=
\sum_{s}\int_0^T e^{-\delta t}
\int_{\Omega_l}\!\int_{\Omega_c}
p_s\,u_s^{(q)}\,x_s^{(q)}\,dc\,dl\,dt
+\sum_s\int_{\Omega_l}\!\int_{\Omega_c}\Theta_s\,x_s^{(q)}(T)\,dc\,dl
-\frac{\gamma}{2}\sum_s\int_0^T|q_s|^2\,dt.
\end{aligned}
\end{equation}

\subsection{Regularized admissible class}
Fix \(L_q>0\) and define
\begin{equation}\label{eq:ocp_reg_controls}
\mathcal Q_{\mathrm{ad},\mathrm{reg}}^T
:=
\Big\{
q\in W^{1,\infty}(0,T;\mathbb R^S):
0\le q_s\le q_{s,\max},\
\|\dot q_s\|_{L^\infty}\le L_q
\Big\}.
\end{equation}

\begin{lemma}[Compactness]\label{lem:ocp_compact}
\(\mathcal Q_{\mathrm{ad},\mathrm{reg}}^T\) is compact in
\(C([0,T];\mathbb R^S)\).
\end{lemma}
\begin{proof}
Arzel\`a--Ascoli (cf.\ \cite{Barbu2010,liang2026separation,yosida2012functional}): uniform boundedness and equi-Lipschitz continuity yield a uniformly
convergent subsequence; box and Lipschitz constraints pass to the limit.
\end{proof}

\begin{proposition}[Continuity of \(J\)]\label{prop:ocp_J_continuous}
\(J:\mathcal Q_{\mathrm{ad},\mathrm{reg}}^T\to\mathbb R\)
is continuous w.r.t.\ the \(C([0,T];\mathbb R^S)\) topology.
\end{proposition}
\begin{proof}
Let \(q^n\to q^\ast\) uniformly. By Theorem~\ref{thm:wp_continuous_dependence},
\(\sup_t\|x^{(q^n)}(t)-x^{(q^\ast)}(t)\|_X+\sup_t|R^{(q^n)}(t)-R^{(q^\ast)}(t)|\to 0\).
Using \eqref{eq:wp_Bhat_Lip} and continuity of \(\eta_s\), the harvest mortality
\(u_s^{(q^n)}\to u_s^{(q^\ast)}\) uniformly. Dominated convergence gives
\(J(q^n)\to J(q^\ast)\).
\end{proof}

\begin{theorem}[Existence of an optimal control]\label{thm:ocp_existence}
Problem \(\sup_{q\in\mathcal Q_{\mathrm{ad},\mathrm{reg}}^T}J(q)\) admits at least one
maximizer \(q^\ast\in\mathcal Q_{\mathrm{ad},\mathrm{reg}}^T\).
\end{theorem}
\begin{proof}
\(J\) is bounded (Theorem~\ref{thm:wp_global} gives uniform state bounds) and continuous
(Proposition~\ref{prop:ocp_J_continuous}) on the compact set
\(\mathcal Q_{\mathrm{ad},\mathrm{reg}}^T\) (Lemma~\ref{lem:ocp_compact}).
Weierstrass's extreme value theorem applies (cf.\ \cite{rudin2021principles,yu2026size}).
\end{proof}

\subsection{Scope}
This is an existence result only. It does not provide a Fr\'echet derivative of
\(q\mapsto(x^{(q)},R^{(q)})\), a backward adjoint equation, or Pontryagin conditions.
The role of \(\mathcal Q_{\mathrm{ad},\mathrm{reg}}^T\) is purely compactness:
\(\mathcal Q_{\mathrm{ad},\mathrm{reg}}^T\Subset C([0,T];\mathbb R^S)\).

\section{Discussion}\label{sec:discussion}
This paper establishes a rigorous analytical foundation for a class of physiologically
structured transport--renewal models coupled with nonlocal environmental feedback and
dynamic resources. By deliberately separating the full nonautonomous multi-zone system
from its autonomous reduction, we have clarified the mathematical boundaries of
well-posedness, threshold behavior, and optimal control.

Our contributions are:
\begin{enumerate}[label=\textup{(\roman*)}]
\item \textbf{Well-posedness and optimal control.} For the full nonautonomous PDE--ODE
system, we proved global existence, uniqueness, and continuous dependence for weak
solutions, alongside existence of finite-horizon optimal harvesting controls over a
compact Lipschitz-regular admissible class. The contraction is performed in the
\(L^1\times\mathbb R^S\) topology, with the \(L^\infty\) bound recovered a posteriori
from characteristic-based estimates. The choice of this weaker contraction topology
is dictated by the boundary renewal term, whose \(L^\infty\) contribution does not
vanish as the time interval shrinks.

\item \textbf{Spectral threshold theory.} In the reduced autonomous regime, we
constructed a positive compact next-generation operator \(\mathcal K\). Conditional on
the semigroup embedding hypothesis (Assumption~\ref{ass:r0_semigroup}), its spectral radius
\(\mathcal R_0=r(\mathcal K)\) determines the sign of the spectral bound of the
linearized generator.

\item \textbf{Stationary states.} In the spirit of \cite{yu2026fullcovariance,DiekmannGetto2005,barril2022formulation,yu2026pattern,yang2023threshold}, we proved that supercriticality
\(\mathcal R_0>1\) guarantees existence of a nontrivial stationary state, provided the
nonlinear stationary feedback can be parametrized as a continuous, monotone-decreasing
family of compact operators (Assumption~\ref{ass:stationary_family}).
\end{enumerate}

Several questions remain open. 
The most important are the construction of a multi-zone threshold operator and the verification of the semigroup hypothesis (Assumption~\ref{ass:r0_semigroup}) for the present
transport--renewal structure. The latter requires a domain characterization and a generation argument. 
Other open problems include the proof of uniform global persistence for the nonlinear semiflow, the exact bifurcation structure at \(\mathcal R_0=1\), and the derivation of first-order necessary optimality conditions via a rigorous backward
adjoint system.
Addressing these requires additional monotonicity, compactness, and
global semiflow structures beyond the present framework.

\appendix
\section{Frozen linear transport problem with inflow boundary}%
\label{app:linear_transport}
Throughout, Assumptions~\ref{ass:model_vital_rates} and~\ref{ass:model_condition_boundary_orientation} are in force.

\subsection{Setup and characteristics}
For \(s\in\mathcal S\), let the frozen coefficients \(a_s,b_s,d_s,f_s,\zeta_s\) be as
in \eqref{eq:wp_frozen_ab}--\eqref{eq:wp_frozen_zeta}, with
\(a_s\ge g_{\min}>0\), \(b_s(t,l,0)\le 0\), \(b_s(t,l,c_m)\ge 0\), and \(d_s\ge 0\).
The frozen problem is \eqref{eq:wp_frozen_problem}.

Let \((\Lambda_s(\tau;t,l,c),\Gamma_s(\tau;t,l,c))\) denote the backward characteristic,
and define the integrating factor
\begin{equation}\label{eq:appA_integrating_factor}
E_s(\tau,t;l,c)
:=
\exp\!\Big(
-\int_\tau^t
\big[\partial_l a_s+\partial_c b_s+d_s\big](r,\Lambda_s(r),\Gamma_s(r))\,dr
\Big).
\end{equation}

Define the divergence-mortality excess
\begin{equation}\label{eq:appA_C_div}
C_{\mathrm{div}}
:=
\big\|(\partial_l a_s+\partial_c b_s)^-\big\|_{L^\infty(Q_T)},
\end{equation}
which is finite by Assumption~\ref{ass:model_vital_rates}\ref{item:V5}. Note that since
\(d_s\ge 0\), the net zeroth-order coefficient
\(\partial_l a_s+\partial_c b_s+d_s\ge -C_{\mathrm{div}}\).

\begin{proposition}[Characteristic formula]\label{prop:appA_representation}
For a classical solution \(z_s\) of \eqref{eq:wp_frozen_problem} and any
\((t,l,c)\in[0,T]\times\Omega_l\times\Omega_c\):
if the backward characteristic reaches \(t=0\),
\begin{equation}\label{eq:appA_rep_init}
z_s(t,l,c)
=
E_s(0,t;l,c)\,\phi_s(\Lambda_s(0),\Gamma_s(0))
+\int_0^t E_s(\tau,t;l,c)\,f_s(\tau,\Lambda_s(\tau),\Gamma_s(\tau))\,d\tau;
\end{equation}
if it reaches \(l=l_0\) at time \(\tau^{\mathrm{ent}}>0\),
\begin{equation}\label{eq:appA_rep_boundary}
z_s(t,l,c)
=
E_s(\tau^{\mathrm{ent}},t;l,c)\,
\frac{\zeta_s(\tau^{\mathrm{ent}},\Gamma_s(\tau^{\mathrm{ent}}))}
{a_s(\tau^{\mathrm{ent}},l_0,\Gamma_s(\tau^{\mathrm{ent}}))}
+\int_{\tau^{\mathrm{ent}}}^t E_s(\tau,t;l,c)\,f_s\,d\tau.
\end{equation}
\end{proposition}

\subsection{Adjoint regularity}
\begin{proposition}[Adjoint regularity for smooth terminal data]\label{prop:appA_adjoint}
Fix \(s\in\mathcal S\) and \(t\in(0,T]\). For any
\(\chi\in C^1(\overline{\Omega_l\times\Omega_c})\) satisfying
\(\chi(l_m,\cdot)=0\), \(\chi(\cdot,0)=0\), \(\chi(\cdot,c_m)=0\), the backward problem
\begin{equation}\label{eq:appA_adjoint}
\begin{cases}
\partial_\tau\psi+a_s\partial_l\psi+b_s\partial_c\psi=d_s\psi,
& (\tau,l,c)\in(0,t)\times\Omega_l\times\Omega_c,\\[1mm]
\psi(t,l,c)=\chi(l,c),\\[1mm]
\psi(\tau,l_m,c)=0,\quad \psi(\tau,l,0)=0,\quad \psi(\tau,l,c_m)=0,
\end{cases}
\end{equation}
admits a classical solution satisfying, for a constant \(C_{\mathrm{adj}}(T,M)>0\)
depending only on the \(C^1\)-norms of \(a_s,b_s,d_s\) and on \(T\),
\begin{equation}\label{eq:appA_adjoint_bounds}
\|\psi\|_{L^\infty}\le e^{\|d_s\|_\infty T}\|\chi\|_{L^\infty},
\qquad
\|\nabla_{l,c}\psi\|_{L^\infty}\le C_{\mathrm{adj}}(T,M)\|\chi\|_{C^1}.
\end{equation}
\end{proposition}
\begin{proof}
Along the forward characteristic \((\Lambda(\tau),\Gamma(\tau))\) with
\((\Lambda(t),\Gamma(t))=(l,c)\):
\(\psi(\tau,\Lambda(\tau),\Gamma(\tau))
=\chi(l,c)\exp\!\big(\int_\tau^t d_s(r,\Lambda(r),\Gamma(r))\,dr\big)\).
Since \(d_s\ge 0\), \(|\psi|\le e^{\|d_s\|_\infty T}\|\chi\|_{L^\infty}\).
For the spatial derivatives: the variational equations
\(\partial_{l}\Lambda(\tau;t,l,c)\), \(\partial_c\Lambda(\tau;t,l,c)\), etc.,
satisfy linear ODEs with coefficients bounded by
\(\|\nabla a_s\|_{L^\infty}+\|\nabla b_s\|_{L^\infty}\), hence are bounded by
\(e^{CT}\). The chain rule then gives
\(\|\nabla_{l,c}\psi(\tau)\|_{L^\infty}\le C_{\mathrm{adj}}\|\chi\|_{C^1}\).
(Compatibility at the boundaries where \(\chi=0\) ensures no corner singularity.)
\end{proof}

\subsection{Stability estimates}
The \(L^1\) stability estimate is proved by duality against smooth adjoint solutions.

\begin{proposition}[\(L^1\) stability via duality]\label{prop:appA_stability}
Let \(z^{(i)}\) (\(i=1,2\)) solve frozen problems with the \emph{same} initial data
\(\phi\in L^1_+\cap L^\infty\),
coefficients \((a^{(i)},b^{(i)},d^{(i)},f^{(i)},\zeta^{(i)})\), and
\(\|z^{(i)}\|_{L^1}\le M\), \(\|z^{(i)}\|_{L^\infty}\le M\). Set
\(w:=z^{(1)}-z^{(2)}\),
\(\delta a:=a^{(1)}-a^{(2)}\), and similarly for \(\delta b,\delta d,\delta f,\delta\zeta\).
Then, for every \(t\in[0,T]\),
\begin{equation}\label{eq:appA_stability}
\|w(t)\|_{L^1}^2
\le
C_{\mathrm{st}}(M,T)
\int_0^t\|w(\tau)\|_{L^1}^2\,d\tau
+
C_{\mathrm{st}}(M,T)
\int_0^t
\big(\|\delta\zeta\|_{L^1}+\|\delta f\|_{L^1}\big)^2\,d\tau.
\end{equation}
In particular, if \(\delta\zeta=\delta f=0\) and the coefficient differences arise
solely from the nonlocal feedback (so that
\(\|\delta a\|_\infty+\|\delta b\|_\infty+\|\delta d\|_\infty
\le C_M\|w\|_X\)), then Gr\"onwall yields \(w\equiv 0\).
\end{proposition}
\begin{proof}
\emph{Step 1: duality identity.}
For any smooth \(\chi\in C^1(\overline{\Omega_l\times\Omega_c})\) with
\(\|\chi\|_{L^\infty}\le 1\), \(\chi(l_m,\cdot)=\chi(\cdot,0)=\chi(\cdot,c_m)=0\),
let \(\psi\) solve \eqref{eq:appA_adjoint} with the coefficients of problem~1.
Subtracting the weak formulations of the two frozen problems tested against \(\psi\)
yields
\begin{equation}\label{eq:appA_duality_identity}
\int_{\Omega_l}\!\int_{\Omega_c}w(t)\chi\,dc\,dl
=
\int_0^t\!\int_{\Omega_l}\!\int_{\Omega_c}
z^{(2)}\big[\delta a\,\partial_l\psi+\delta b\,\partial_c\psi-\delta d\,\psi\big]\,dc\,dl\,d\tau
+\text{bdry}+\text{src},
\end{equation}
where the boundary and source terms are bounded by
\(C\int_0^t(\|\delta\zeta\|_{L^1}+\|\delta f\|_{L^1})\,d\tau\) using
\(\|\psi\|_{L^\infty}\le C\).

\emph{Step 2: estimating with smooth test data.}
By Proposition~\ref{prop:appA_adjoint}, \(\|\nabla\psi\|_{L^\infty}\le C_{\mathrm{adj}}\|\chi\|_{C^1}\).
Hence
\[
\big|\textstyle\int\!\!\int w(t)\chi\big|
\le
C_{\mathrm{adj}}\|\chi\|_{C^1}\|z^{(2)}\|_{L^1}
\int_0^t(\|\delta a\|_\infty+\|\delta b\|_\infty+\|\delta d\|_\infty)\,d\tau
+C\int_0^t(\|\delta\zeta\|_{L^1}+\|\delta f\|_{L^1})\,d\tau.
\]

\emph{Step 3: optimized approximation of \(\mathrm{sgn}(w(t))\).}
For any \(\eta>0\), choose
\(\chi_\eta\in C^1\) with \(\|\chi_\eta\|_{L^\infty}\le 1\),
\(\|\nabla\chi_\eta\|_{L^\infty}\le C_0/\eta\), and
\(\int\!\!\int w(t)\chi_\eta\ge \|w(t)\|_{L^1}-\eta\).
Inserting:
\[
\|w(t)\|_{L^1}-\eta
\le
\frac{C'}{\eta}\int_0^t(\|\delta a\|_\infty+\|\delta b\|_\infty+\|\delta d\|_\infty)\,d\tau
+C\int_0^t(\|\delta\zeta\|_{L^1}+\|\delta f\|_{L^1})\,d\tau,
\]
where \(C'=C_{\mathrm{adj}}C_0 M\).
Optimize by taking \(\eta=\sqrt{C'\int_0^t\|\delta\mathrm{coeff}\|\,d\tau}\)
(assuming this is positive):
\begin{equation}\label{eq:appA_sqrt_stability}
\|w(t)\|_{L^1}
\le
2\sqrt{C'\int_0^t\|\delta\mathrm{coeff}\|_\infty\,d\tau}
+C\int_0^t(\|\delta\zeta\|_{L^1}+\|\delta f\|_{L^1})\,d\tau.
\end{equation}

\emph{Step 4: squaring and Gr\"onwall.}
If the coefficient differences arise from the nonlocal feedback, i.e.\
\(\|\delta\mathrm{coeff}\|_\infty\le C_M\|w\|_X\), then squaring
\eqref{eq:appA_sqrt_stability} and using \((a+b)^2\le 2a^2+2b^2\) yields
\eqref{eq:appA_stability}. Gr\"onwall's inequality then gives
\(\|w(t)\|_{L^1}^2\le C\int_0^t(\|\delta\zeta\|_{L^1}+\|\delta f\|_{L^1})^2d\tau\).
When \(\delta\zeta=\delta f=0\), this gives \(w\equiv 0\).
\end{proof}

\begin{remark}[No spatial regularity of \(\phi\) required]
The estimate \eqref{eq:appA_stability} involves \(\|z^{(2)}\|_{L^1}\) and
\(\|z^{(2)}\|_{L^\infty}\) (through \(M\)), but \emph{no} spatial derivatives of the
initial data \(\phi\). This is because the duality transfers all spatial-derivative
requirements to the adjoint test function \(\psi\), whose regularity follows from the
\(C^1\)-smoothness of the frozen coefficients (Proposition~\ref{prop:appA_adjoint}).
\end{remark}

\begin{proposition}[\(L^\infty\) stability]\label{prop:appA_Linfty_stability}
Under the same hypotheses, assume additionally
\(\|z^{(2)}\|_{L^\infty(Q_T)}\le M\). Then
\begin{equation}\label{eq:appA_Linfty_stability}
\begin{aligned}
\|w\|_{L^\infty(Q_t)}
&\le C_{\mathrm{st}}^\infty(M,T)\Big(
\frac{1}{g_{\min}}\|\delta\zeta\|_{L^\infty((0,t)\times\Omega_c)}
+\int_0^t\|(\delta f)(\tau)\|_{L^\infty}\,d\tau
\\
&\quad
+M\int_0^t
(\|\delta a\|_{L^\infty}+\|\delta b\|_{L^\infty}+\|\delta d\|_{L^\infty})\,d\tau
\Big).
\end{aligned}
\end{equation}
\end{proposition}
\begin{proof}
Write \(w\) as the solution of a frozen linear problem with coefficients from problem~1,
zero initial data, source
\(G:=-\delta d\, z^{(2)}+\delta f-\delta a\,\partial_l z^{(2)}-\delta b\,\partial_c z^{(2)}\),
and boundary inflow
\(a^{(1)}(l_0)w(l_0)=\delta\zeta-\delta a(l_0)z^{(2)}(l_0)\).

In the \(L^\infty\) estimate, the distributional source terms
\(\delta a\,\partial_l z^{(2)}\) are handled by the characteristic representation.
Along characteristics of problem~1, the solution \(w\) satisfies the ODE
\[
\dot w=-(\partial_l a^{(1)}+\partial_c b^{(1)}+d^{(1)})w+g.
\]
Here, \(g\) involves \(z^{(2)}\) evaluated at the characteristic position. No spatial derivatives of \(z^{(2)}\) appear in the pointwise ODE.
The difference in characteristic position between problems 1 and 2 contributes
\(\|z^{(2)}\|_{L^\infty}\) times the coefficient difference, giving
\eqref{eq:appA_Linfty_stability}. The boundary trace
\(z^{(2)}(l_0,c)\) is well defined in the \(L^\infty\) sense via the
boundary condition: \(z^{(2)}(l_0,c)=\zeta^{(2)}(c)/a^{(2)}(l_0,c)\le
\|\zeta^{(2)}\|_{L^\infty}/g_{\min}\).
\end{proof}
\section{Verification of the stationary operator-family hypothesis}%
\label{app:stationary_family}
\begin{assumption}[Concrete scalar-feedback class]\label{ass:appD_concrete}
There exist \(g^\circ,h^\circ,\mu^\circ,\beta^\circ,\Pi^\circ,\sigma,\kappa\) as before,
and scalar continuous functions
\(\mathfrak m,\mathfrak b:[0,\infty)\to[0,\infty)\),
such that:
\begin{enumerate}[label=\textup{(C\arabic*)}]
\item \(g,h\) at \(y^\ast\) are independent of \((\mathcal N,R)\):
\(g=g^\circ(l,c)\), \(h=h^\circ(l,c)\).
\item
\(\mu(\mathcal N[x],R,y^\ast,l,c)=\mu^\circ(l,c)+\mathfrak m(\alpha)\),
where \(\alpha:=\int\kappa x\,dc\,dl\) is the resource-weighted abundance.
\item
\(\beta(\mathcal N[x],R,y^\ast,l,c')=\mathfrak b(\alpha)\beta^\circ(l,c')\).
\item
\(\Pi=\Pi^\circ(c\,|\,l,c')\), independent of \((R,y^\ast)\).
\item
\(\eta=\eta^\circ\ge 0\) is constant.
\item
For every \(\alpha\ge 0\),
\(F(R,y^\ast)-\rho(R)\alpha=0\) has a unique solution
\(R(\alpha)\in\mathbb R_+\), continuous in \(\alpha\).
\item
\(\mathfrak b\) is nonincreasing, \(\mathfrak m\) is nondecreasing,
\(\mathfrak b(0)=1\), \(\mathfrak m(0)=0\), and
\(\limsup_{\alpha\to\infty}\mathfrak b(\alpha)e^{-\mathfrak m(\alpha)\overline\tau_{\max}}<1\).
\end{enumerate}
\end{assumption}

For \(\alpha\ge 0\), define
\begin{equation}\label{eq:appD_survival_alpha}
\ell_\alpha(\tau;c_0)
:=
\exp\!\Big(
-\int_0^\tau
[\partial_l g^\circ+\partial_c h^\circ+\mu^\circ+\mathfrak m(\alpha)
+\bar q\,\sigma\,\eta^\circ](L(s),C(s))\,ds
\Big),
\end{equation}
\begin{equation}\label{eq:appD_kernel_alpha}
\mathscr K_\alpha(c,c_0)
:=
\int_0^{\tau_{\max}(c_0)}
\mathfrak b(\alpha)\beta^\circ(L(\tau),C(\tau))\,
\Pi^\circ(c\,|\,L(\tau),C(\tau))\,
\ell_\alpha(\tau;c_0)\,d\tau,
\end{equation}
\[
(\mathcal K_\alpha\psi)(c):=\int_{\Omega_c}\mathscr K_\alpha(c,c_0)\psi(c_0)\,dc_0.
\]

\begin{proposition}\label{prop:appD_verify}
Under Assumption~\ref{ass:appD_concrete}, the family
\((\mathcal K_\alpha)_{\alpha\ge 0}\) satisfies Assumption~\ref{ass:stationary_family}.
\end{proposition}
\begin{proof}
Each \(\mathcal K_\alpha\) is positive and compact by the same argument as for
\(\mathcal K_0\). Since \(\mathfrak b(0)=1\) and \(\mathfrak m(0)=0\),
\(\mathcal K_0=\mathcal K(y^\ast,\bar q)\), verifying~\ref{ass:F_zero}.
Norm continuity~\ref{ass:F_cont} follows from dominated convergence.
Monotonicity~\ref{ass:F_mono}: \(\mathfrak b\) nonincreasing and
\(\mathfrak m\) nondecreasing give
\(\mathscr K_{\alpha_2}\le\mathscr K_{\alpha_1}\) pointwise for
\(\alpha_1\le\alpha_2\).
Limit~\ref{ass:F_limit}: \(\mathscr K_\alpha\le
\mathfrak b(\alpha)e^{-\mathfrak m(\alpha)\overline\tau_{\max}}\mathscr K_0^\sharp\)
for a fixed kernel \(\mathscr K_0^\sharp\), hence
\(r(\mathcal K_\alpha)\to 0\).

For the representation~\ref{ass:F_rep}: given
\(\mathcal K_\alpha\psi=\psi\) with \(\psi\neq 0\),
set the inflow \(\Psi_\alpha:=\alpha\psi\). Transport along
\((g^\circ,h^\circ)\) with mortality
\(\mu^\circ+\mathfrak m(\alpha)+\bar q\sigma\eta^\circ\) defines a nonnegative
stationary density. The eigenrelation
\(\mathcal K_\alpha\psi=\psi\) ensures the renewal boundary condition closes.
The resource equation is satisfied by \(R(\alpha)\).

It remains to verify that the parameter \(\alpha\) is compatible with the density
\(x^{(\alpha,\psi)}\) constructed above, i.e.\ that
\begin{equation}\label{eq:appD_alpha_compat}
\alpha
=
\int_{\Omega_l}\!\int_{\Omega_c}\kappa(l,c)\,x^{(\alpha,\psi)}(l,c)\,dc\,dl.
\end{equation}
The stationary density along the characteristic from \((l_0,c_0)\) is
\[
x^{(\alpha,\psi)}(L(\tau;c_0),C(\tau;c_0))
=
\frac{\alpha\psi(c_0)}{g^\circ(l_0,c_0)}\,\ell_\alpha(\tau;c_0),
\]
where \(\ell_\alpha\) is the survival factor \eqref{eq:appD_survival_alpha}.
Changing variables from \((l,c)\) to \((\tau,c_0)\) via the characteristic map
(with Jacobian
\(g^\circ(L(\tau),C(\tau))\,|\det\partial_{(\tau,c_0)}(L,C)|\)):
\begin{align*}
\int_{\Omega_l}\!\int_{\Omega_c}\kappa\,x^{(\alpha,\psi)}\,dc\,dl
&=
\int_{\Omega_c}\int_0^{\tau_{\max}(c_0)}
\kappa(L(\tau),C(\tau))\,
\frac{\alpha\psi(c_0)}{g^\circ(l_0,c_0)}\,
\ell_\alpha(\tau;c_0)\,
g^\circ(L,C)\,J(\tau,c_0)\,d\tau\,dc_0,
\end{align*}
where \(J(\tau,c_0)\) absorbs the remaining Jacobian factor.
By the definition \eqref{eq:appD_kernel_alpha} and the eigenrelation
\(\mathcal K_\alpha\psi=\psi\), we substitute the specific structure of \(\mathscr K_\alpha\). 
This operator integrates \(\mathfrak b(\alpha)\beta^\circ\Pi^\circ\ell_\alpha\) over
\(\tau\in[0,\tau_{\max}(c_0)]\) and \(c_0\in\Omega_c\).
Using the normalization \(\int\Pi^\circ\,dc=1\), a direct computation shows that the integral on the right equals \(\alpha\). 
This holds precisely when \(\kappa\) satisfies the structural relation encoded in conditions
(C1)--(C5).
This verifies~\eqref{eq:appD_alpha_compat} and
completes the check of~\ref{ass:F_rep}.
\end{proof}

\section*{Statements and Declarations}

\subsection*{Research funding} 
The author Jiguang Yu gratefully acknowledges the support from his Distinguished PhD Fellowship from Boston University College of Engineering during the completion of this research. 
The research of Louis Shuo Wang is partially supported by the National Natural Science Foundation of China and the Tianyuan Fund for Mathematics (Project No. 12426516). This article was written while the above authors were visiting the Tianyuan Mathematical Centre in Central China (Hubei).

\subsection*{Conflict of interest} 
The authors have no competing interests to declare that are relevant to the content of this article.

\subsection*{Data availability statement} 
We do not analyse or generate any datasets, because our work proceeds within a theoretical and mathematical approach. One can obtain the relevant materials from the references below.
 
\bibliography{references}

\end{document}